\documentclass[12pt,a4paper]{article}
\usepackage[utf8]{inputenc}
\usepackage[letterpaper,left=1in,right=1in,top=1in,bottom=1in]{geometry}
\usepackage{authblk}

\usepackage{graphicx} 
\graphicspath{{./figs/}}      
\usepackage{epsfig} 
\usepackage{mathptmx} 
\usepackage{times} 
\usepackage{amsmath} 
\usepackage{amssymb}  
\usepackage{amsfonts}
\usepackage{color}
\usepackage{verbatim}
\usepackage{epsf,wrapfig}
\usepackage{subfigure}
\usepackage{subfiles}
\usepackage{url}
\usepackage[square,sort,numbers]{natbib}

\usepackage{todonotes}
\usepackage{epstopdf}
\DeclareMathAlphabet{\mathpzc}{OT1}{pzc}{m}{it}

\newtheorem{theorem}{Theorem}[section]

\newtheorem{conjecture}{Conjecture}[section]

\newtheorem{definition}{Definition}[section]

\renewcommand{\d}[2]{\frac{d #1}{d #2}} 
 
\newcommand{\bigO}{\ensuremath{\mathcal{O}}}

\newcommand{\subjclass}[1]{MSC2010 numbers: #1}
\newcommand{\keywords}[1]{Key words and phrases: #1}

\makeatletter
\def\blx@maxline{77}
\makeatother

\newlength{\spc} 

\title{Computational Method for Phase Space Transport with Applications to Lobe Dynamics and Rate of Escape}

\author[1]{Shibabrat Naik}
\author[2]{Francois Lekien}
\author[3]{Shane D. Ross}

\affil[1,3]{Engineering Mechanics Program\\ 
	Department of Biomedical Engineering \& Mechanics, Virginia Tech\\ Blacksburg, VA-24061, USA}
\affil[2]{\'Ecole Polytechnique, Universit\'e Libre de Bruxelles\\ B-1050 Brussels, Belgium}

\date{}
\begin{document}


%
%



\maketitle

\footnotetext[1]{E-mail: shiba@vt.edu}
\footnotetext[2]{E-mail: lekien@ulb.ac.be}
\footnotetext[3]{E-mail: sdross@vt.edu}

\begin{abstract}
	Lobe dynamics and escape from a potential well are general frameworks introduced to study phase space transport in chaotic dynamical systems. 
	While the former approach studies how regions of phase space are transported by reducing the flow to a two-dimensional map, the latter approach studies the phase space structures that lead to critical events by crossing periodic orbit around saddles.  
	Both of these frameworks require computation with curves represented by millions of points---computing intersection points between these curves and area bounded by the segments of these curves---for quantifying the transport and escape rate.
	We present a theory for computing these intersection points and the area bounded between the segments of these curves based on a classification of the intersection points using equivalence class.
	We also present an alternate theory for curves with nontransverse intersections and a method to increase the density of points on the curves for locating the intersection points accurately.
	The numerical implementation of the theory presented herein is available as an open source software called \emph{Lober}. We used this package to demonstrate the application of the theory to lobe dynamics that arises in fluid mechanics, and rate of escape from a potential well that arises in ship dynamics.
\end{abstract}



\medskip

\bigskip

\section{Introduction}
In chaotic dynamical systems, phase space transport is an approach for understanding and characterizing how regions of phase space move over time. 
This approach is used widely in studying escape and transition rate in classical mechanics and dynamical astronomy~\cite{Koon2000a,Dellnitz2005,Zotos2016a,Zotos2016b}, loss of global integrity in ship capsize~\cite{McRobie1991}, reaction and dissociation rate in chemical reactions~\cite{Martens1987,Gillilan1991,Toda1995}, chaotic advection in fluid mechanics~\cite{Aref1984,Aref1988}, wake generation behind a cylinder in a fluid flow~\cite{duan1999lagrangian}, and transport in geophysical flow~\cite{samelson2006lagrangian}.
Furthermore, phase space transport is also relevant for prediction, control, and design in a myriad of natural processes and engineering systems; see Ref.~\cite{Ross2012} for an overview. 
In this approach, lobe dynamics, introduced in Ref.~\cite{Rom-kedar1990}, is a geometric framework for studying the global transport in 1 and 2 degree-of-freedom (DOF) Hamiltonian systems that can be reduced to two-dimensional maps. 
Moreover, there has been few attempts at extending lobe dynamics to higher dimensional Hamiltonian systems with small perturbation where the existence of normally hyperbolic invariant set and its associated codimension-1 stable and unstable manifolds can be guaranteed; see Refs.~\cite{Wiggins1990,Gillilan1991,beigie1995multiple,Beigie1995,lekien2007lagrangian} for details. 
In this article, we present a theory and numerical methods which are motivated by lobe dynamics in 2-DOF Hamiltonian systems that can be reduced to two-dimensional map by using suitable Poincar\'e section, or return map, or time-T map. 
Hence, the methods developed herein are applicable to a diverse array of problems in physical and engineering sciences. 

Following the developments in \cite{Rom-kedar1990,Rom-kedar1990a}, lobe dynamics can be stated as a systematic study of a fates and histories of a set of initial conditions. Formally speaking, the two-dimensional phase space $M$ of a Poincar\'e map $f$ (in general, this can a return map or a time-T map) can be partitioned into regions with boundaries consisting of parts of the boundary of $M$ (which may be at infinity) and/or segments of stable and unstable invariant manifolds of hyperbolic fixed points, $p_i, i=1,...,N$, as shown schematically in Fig.~\ref{fig:regions}(a). When the manifolds $W^u_{p_i}$ and $W^s_{p_j}$ are followed out on a global scale, they often intersect in primary intersection points or \emph{pips}, for example $\{ q_1, q_2, q_3, q_4, q_5, q_6 \}$ as in Fig.~\ref{fig:regions}(b), and secondary intersection points or \emph{sips}, for example $\{ q_7,q_8,q_9,q_{10},q_{11},q_{12} \}$ in Fig.~\ref{fig:regions}(b). More precisely, an intersection point is called a pip if the curves connecting the point and the hyperbolic points intersect only at that point, and is called a sip otherwise. These intersections allow one to define boundaries between regions $\{R_i\}$, as illustrated in Fig.~\ref{fig:regions}(b). Moreover, the transport between regions of phase space can be completely described by the dynamical evolution of parcels of phase space enclosed by segments of the stable and unstable manifolds called \emph{lobes}.

\begin{figure}[!ht]
	\begin{tabular}{cc}
		\includegraphics[width=0.45\textwidth]{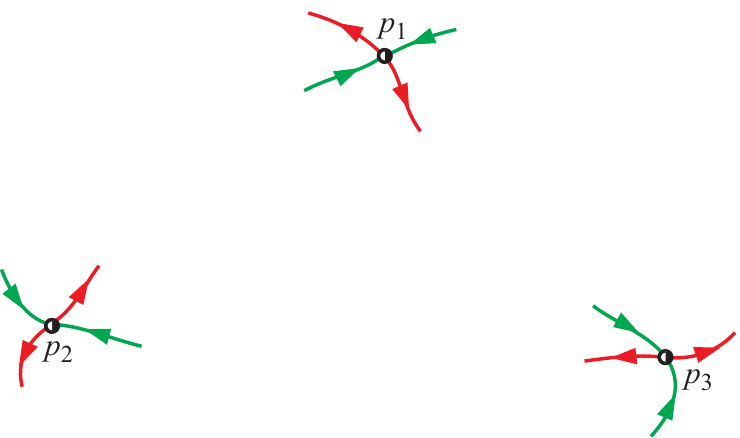} &
		\includegraphics[width=0.45\textwidth]{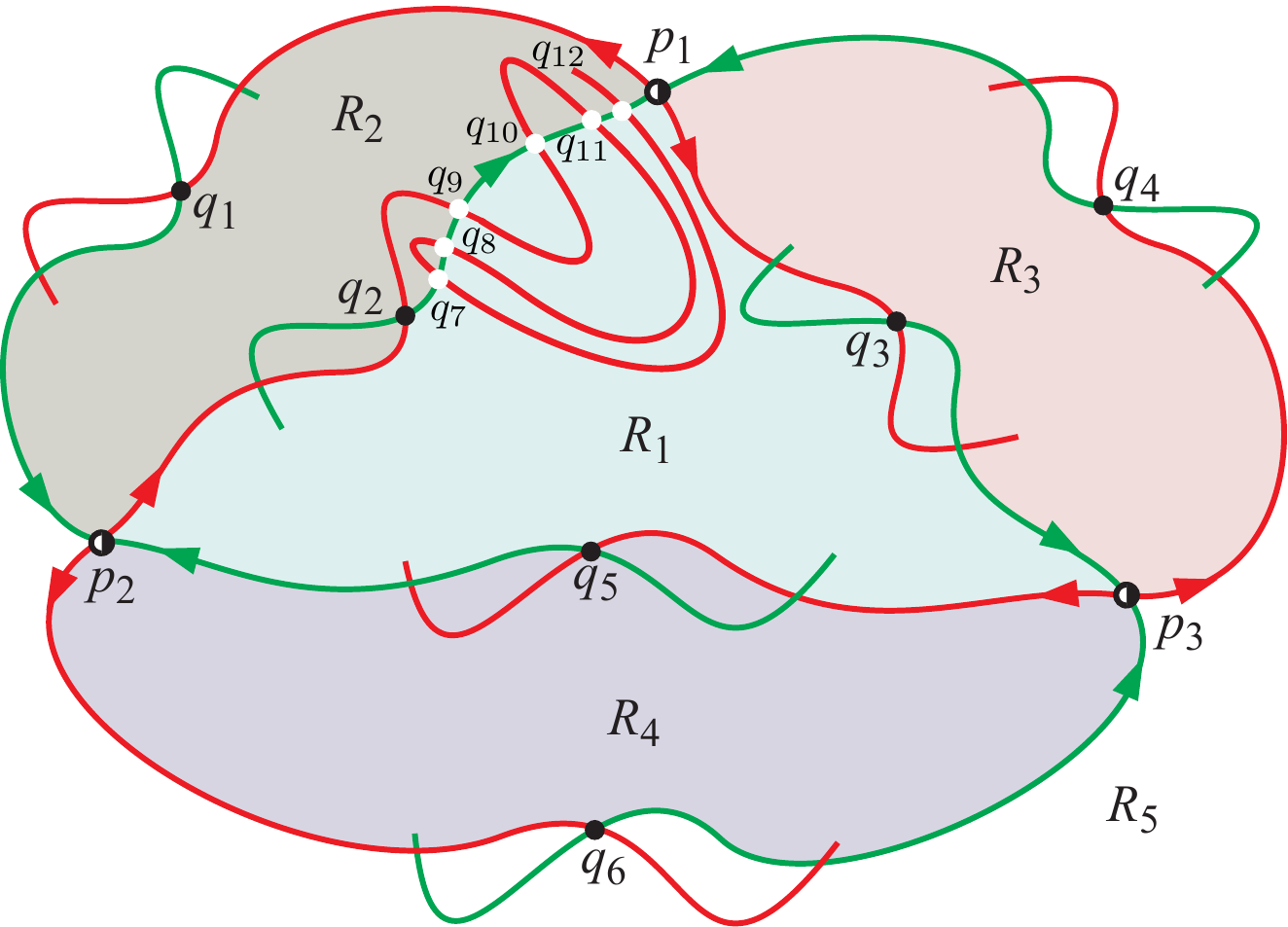}\\
		{\footnotesize (a)} & {\footnotesize (b)}  
	\end{tabular}
	\caption{(a) Pieces of the local unstable and stable manifolds, $W^u(p_i)$ (red) and $W^s(p_i)$ (green) of saddle fixed points $\{p_i\}$. 
		(b) When the manifolds $W^u(p_i)$ and $W^s(p_i)$ are followed out on a global scale, they often intersect in primary intersection points $\{q_i\}$. These intersections allow one to define boundaries between regions $\{R_i\}$.
	}
	\label{fig:regions}       
\end{figure}



Classically, invariant manifolds of hyperbolic fixed points are computed for as long as possible and lobes are extracted from the two curves \cite{Rom-kedar1990,Dellnitz2005}. However, it is possible to get more lobes, hence compute transport for a much longer period of time, by integrating the boundaries
directly, particularly in the complex case of multiple, self-intersecting, lobes~\cite{Dellnitz2005}.
In this approach, the notion of pips and sips are not suitable, and thus we propose a generalization of definition of lobes for two intersecting closed curves.
When applying lobe dynamics to transport problems, it is of eventual interest to quantify volume of phase space $M$ that crosses the boundary with higher iterates of the Poincar\`e map $f$. This is typically achieved in multiple ways:
\begin{itemize}
	\item Distribute test points inside the phase space region of interest and computing the iterate/s required for each to escape. This will also need to take care of the re-entrainment due to the underlying turnstile (a pair of lobe that moves points across a boundary) mechanism~\cite{Mackay1984,Rom-kedar1990,Wiggins1992chaotic,RomKedar1999}. 
	
	\item Constructing a functional of the nonlinear system of vector field which measures area between the manifolds as parametrized by time. This is a semi-analytical approach and only remains valid near small perturbations, for example Melnikov method for mixing, stirring, optimal phase space flux, and action-integral method~\cite{Meiss1992,Sandstede2000,Balasuriya2005,Balasuriya2006,Mosovsky2011,Balasuriya2014}. 
	
	\item Following the boundaries of separatrices/manifolds as it is evolved in time and compute set operations with its pre-images/images. This is a more general but surely a difficult approach due to the stretching and folding of the curves that are involved in such computations~\cite{Mitchell2003,Mitchell2003a,Mitchell2006}.
\end{itemize}
%

For concreteness, let us consider the lobes formed due to a heteroclinic tangle as shown in Fig.~\ref{fig:regions}(b) or due to a homoclinic tangle as shown in Fig.~\ref{fig:lobes_pips_sips}. In both cases, the lobes are regions of phase space bounded by curves oriented in opposite directions, this is because the segments bounding the lobes are stable and unstable manifold of the hyperbolic fixed points, $p_1$ and $p_2$ (Fig.~\ref{fig:regions}(b)), or hyperbolic fixed point $p$ (Fig.~\ref{fig:lobes_pips_sips}). Topologically, the lobe areas can be represented by the difference of the area bounded by the individual curve. For example, let us consider two closed curves $C_1$ and $C_2$ in $\mathbb{R}^2$, oriented in counter-clockwise and clockwise direction bounding the two-dimensional subsets of $\mathbb{R}^2$, $A_1$ and $A_2$, respectively, as in Fig.~\ref{fig:lobex_separated}, then the lobe area can be represented by the set difference of the area, that is $A_2 \setminus A_1$ as shown in Fig.~\ref{fig:A2minusA1_shaded}. Thus, studying phase space transport using lobe dynamics boils down to the fundamental problem of \emph{determining the area $A_1\setminus A_2$ and $A_2 \setminus A_1$ defined by these two curves}. We note that the area $A_1 \setminus A_2$ can be computed once $A_2 \setminus A_1$ is known by using
\begin{equation}
A_1 \setminus A_2 = A_1 - (A_2 - A_2 \setminus A_1)
\end{equation}
So, we approach this problem by developing methods that can calculate the area bounded by closed curves which are typically generated by intersecting manifolds in phase space transport.
\begin{figure}[!ht]
	\centering
	\subfigure[]{\includegraphics[width=0.3\textwidth]{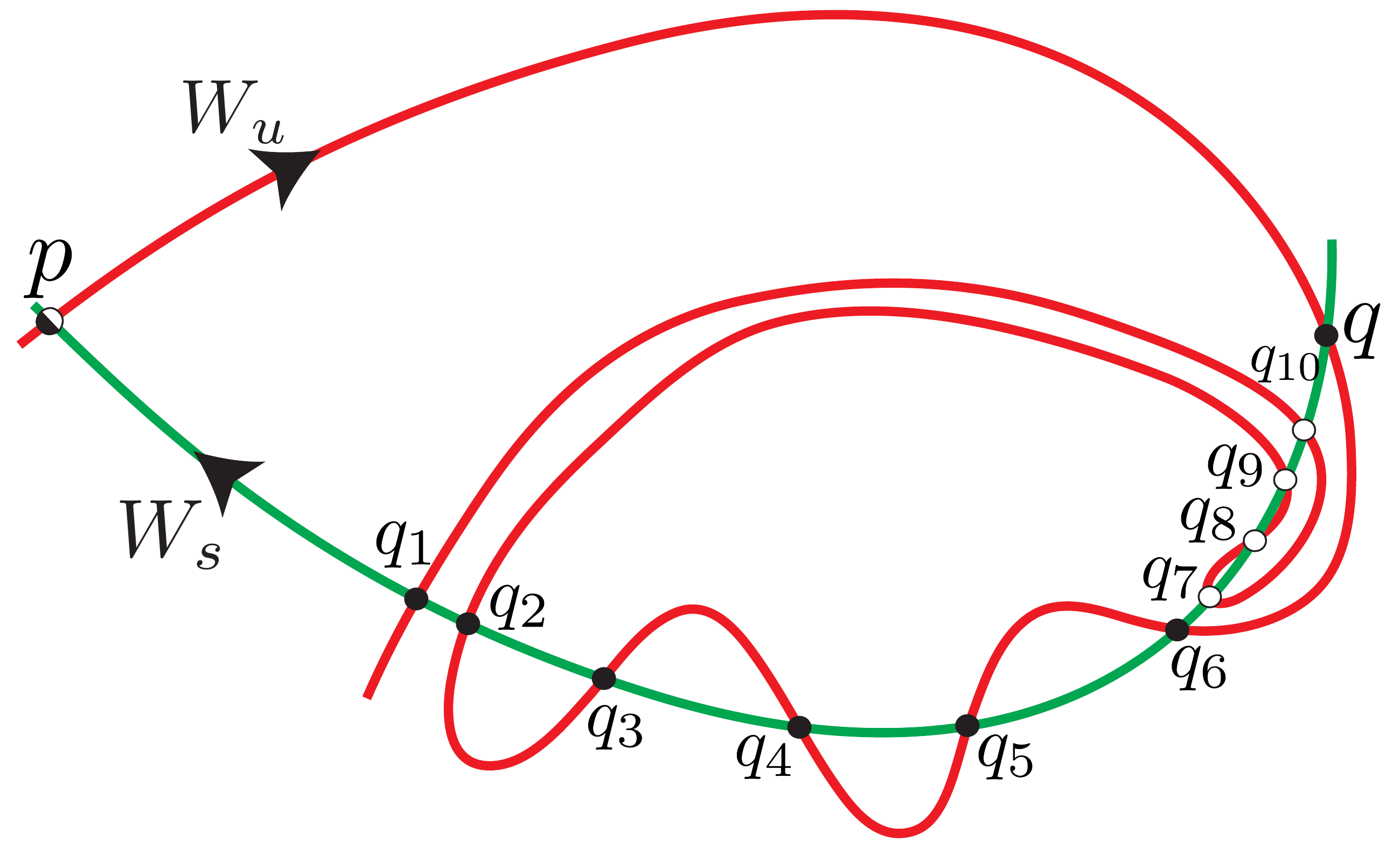}\label{fig:lobes_pips_sips}}
	\subfigure[]{\includegraphics[width=0.38\textwidth]{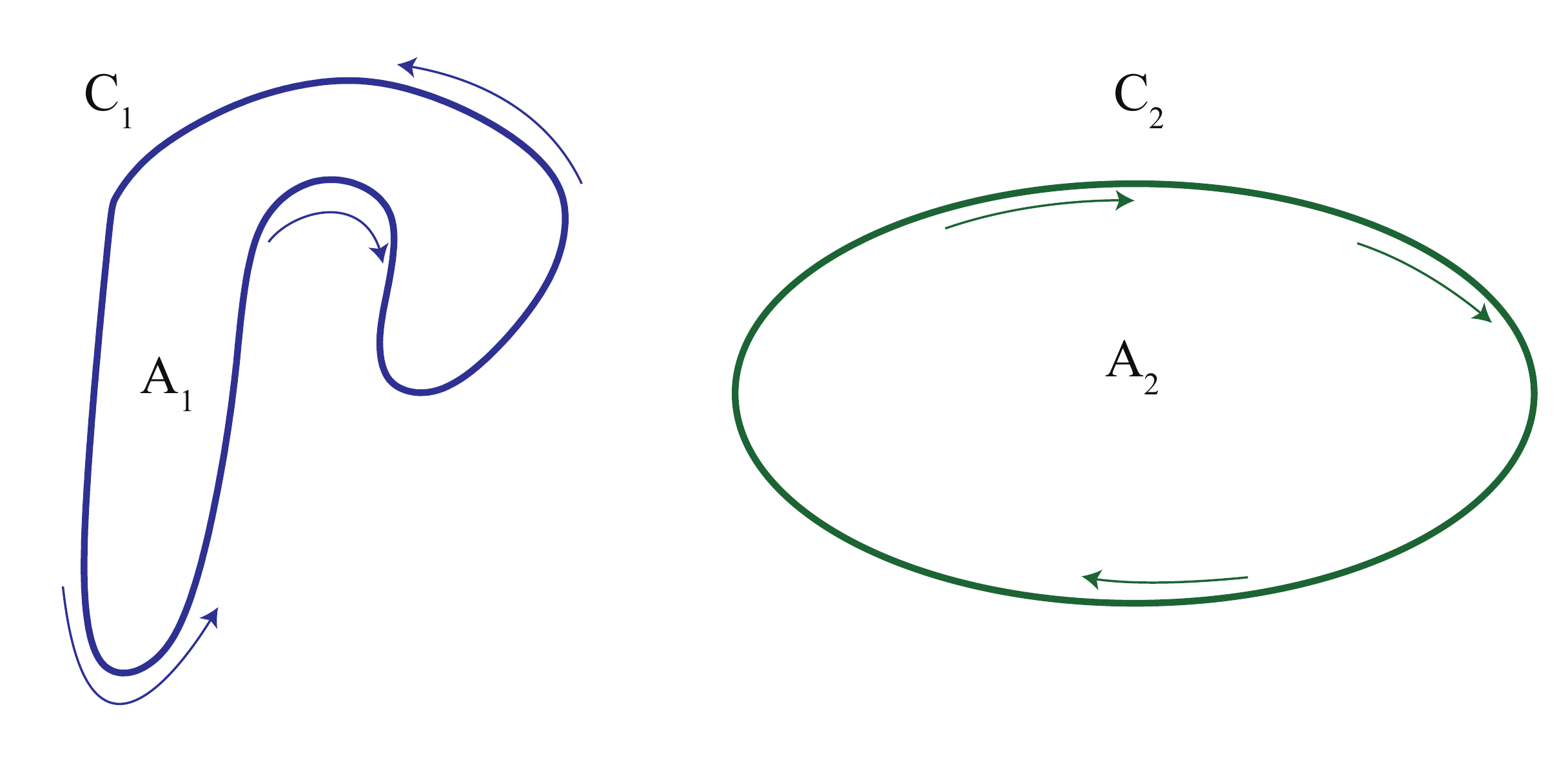}\label{fig:lobex_separated}}
	\subfigure[]{\includegraphics[width=0.22\textwidth]{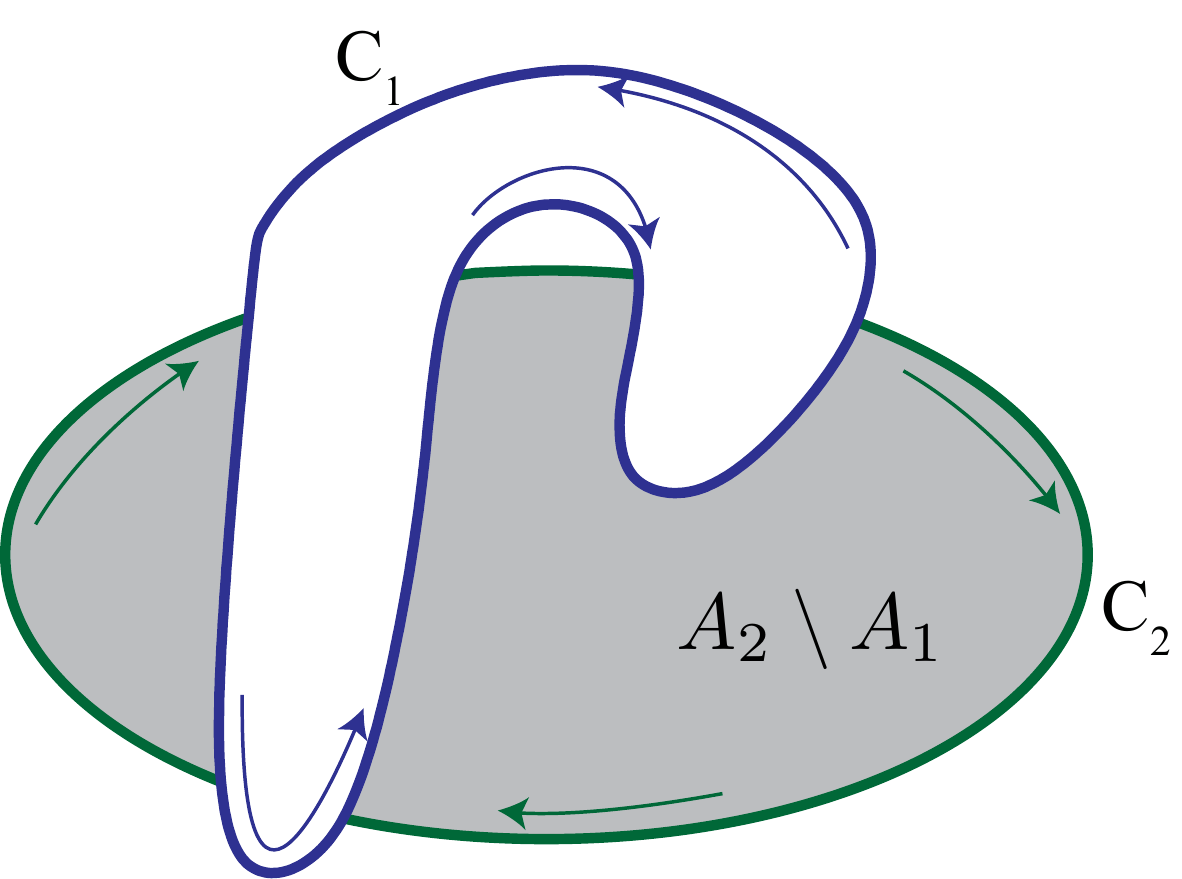}\label{fig:A2minusA1_shaded}}
	\caption{~\protect\subref{fig:lobes_pips_sips} Shows schematically lobes, pips and sips that are formed in a homoclinic tangle. ~\protect\subref{fig:lobex_separated} Shows an example of two closed curves, $C_1$ and $C_2$ in $\mathbb{R}^2$, oriented in counter-clockwise and clockwise direction which bound the two-dimensional subsets of $\mathbb{R}^2$, $A_1$ and $A_2$, respectively.~\protect\subref{fig:A2minusA1_shaded} Shows schematically the area $A_2 \setminus A_1$ which represents a lobe area.}
\end{figure}
%

The main purpose of the paper is to expound on mathematical techniques useful for transport in systems that can be reduced to two-dimensional maps. We present examples in the context of time-periodic and autonomous Hamiltonian systems, however, the methods are applicable in finite-time situations such as numerical simulation and experiments in which one wants to compute the amount of fluid transported from one region of phase space to another.
In the \S~\ref{sect:curve-area}-\ref{sect:inter-pts-lobes}, we derive the theory and numerical methods to separate the intersection points between the two curves into equivalence classes and derive the area of the lobes defined by the two curves. In \S~\ref{sect:non-trans-inter}, we present an alternative method that can be used when portions of the two curves have non-transverse intersections, which arises in the case of an iterated boundary parametrized using an intersection point. In \S~\ref{sect:applications}, the numerical methods implemented as a software package \emph{Lober} are applied to two example problems of interest in the dynamical systems literature: chaotic fluid transport in oscillating vortex pair flow and escape from a potential well in capsize of a ship. 


\section{Curve Area}\label{sect:curve-area}
\subsection{One-Dimensional Integrals for Areas}

Let us define $A_1 = int(C_1)$ and $A_2 = int(C_2)$ as the area enclosed by $C_1$ and $C_2$, respectively and denote this using the standard Lebesgue measure for $\mu$ in $\mathbb{R}^2$ as $[A] = \mu(A)$. The area of each region can be computed as:

\begin{equation}
[A_i] = \int \! \! \! \int _{A_i} dA = \frac{1}{2} \oint _{C_i} y dx - x dy \; ,
\label{eq1darea}
\end{equation} by applying Green's theorem to the vector field
\begin{equation}
\bar{f} = \left( \begin{array}{c}
x\\
y
\end{array} \right) \Longrightarrow \nabla \cdot \bar{f} = 2 \; .
\end{equation}
Eq.~(\ref{eq1darea}) allows us to reduce the computation of the area of a complicated region to a one-dimensional integral over its boundary.
Notice that the sign of the integral is to be reversed if the curves are oriented clockwise. Our hypothesis stating that the curves are oriented in a counter-clockwise direction is equivalent to
\begin{equation}
\forall i : \oint _{C_i} y dx - x dy \ge 0 .
\end{equation}

\subsection{Numerical Methods}
\label{contourint}
We assume that the curves $C_1$ and $C_2$ are given in terms of a sequence of points $(x_i,y_i)$. We want an exact evaluation of the integral in Eq.~\eqref{eq1darea} for curves given by piecewise linear segments connecting $(x_i,y_i)$ and $(x_{i+1},y_{i+1})$. By defining
\begin{equation}
\left\{ \begin{array}{l}
x=x_i + t (x_{i+1} - x_i) \; , \\
y=y_i + t (y_{i+1} - y_i) \; ,
\end{array} \right.
\end{equation} 
as parametric form of the segment, we have
\begin{equation}
\int _{x_i,y_i}^{x_{i+1},y_{i+1}} y dx - x dy = \int _{0} ^{1} (y_i x_{i+1} - x_i y_{i+1}) dt = y_i x_{i+1} - x_i y_{i+1} \; .
\end{equation} 
As a result the exact value of the contour integral for a polygon is given by
\begin{equation}
[A_i] = \frac{1}{2} \oint _{C_i} y dx - x dy = \frac{1}{2} \sum\limits _{i=1}^{i=N} \left( y_i x_{i+1} - x_i y_{i+1} \right) \; .
\label{intpoly}
\end{equation}
It is to be noted that for computing the contour integral the point $i=1$ is repeated as $i={N+1}$, so that the curve is closed, and Eqn.~\ref{intpoly} represents a contour integral. This numerical method forms the core of the software package \emph{Lober}~\footnotemark[1] along with the implementation of the numerical method derived in \S~\ref{sect:curve-area},\ref{sect:inter-pts-lobes}, and \ref{sect:non-trans-inter}.

\footnotetext[1]{This is available as an open-source repository in Github at \url{https://github.com/Shibabrat/curve_densifier} with implementation in C and additional wrapper scripts in MATLAB.}

\section{Intersection Points and Lobe area}\label{sect:inter-pts-lobes}

Lobe dynamics is based on the geometry of a stable manifold, $W^s_{p_+}$, and an unstable manifold, $W^u_{p_-}$, of hyperbolic fixed points, $p_+$ and $p_-$, their intersection points, and areas enclosed by the segments of the invariant manifolds. We note that when $p_+ = p_- = p$, $p$ is called homoclinic point, and when $p_+ \neq p_-$, $p_-$ and $p_+$ are called heteroclinic points. Following the definitions in~\cite{Rom-kedar1990,Wiggins1990}, a point $q_i \in W^u_p \cap W^s_p$ is called a \textit{primary intersection point} (pip) if the segment $U[q_i,p]$ on $W^u_p$ connecting $p$ and $q_i$ and the segment $S[q_i,p]$ on $W^s_p$ connecting $p$ and $q_i$ intersects only at $q_i$, other than the point $p$. A point $q_i$ is a \textit{secondary intersection point} (sip) if there are other intersection points on the segment $U[q_i,p]$ and $S[q_i,p]$. If $q_1$ and $q_2$ are two adjacent pips, then the area enclosed by the segments $U[q_2,q_1]$ and $S[q_2,q_1]$ is called a \textit{lobe}. This is shown in Fig.~\ref{fig:lobes_pips_sips}, where $\{q_1,\ldots,q_6\}$ are pips, and $\{q_7,\ldots,q_{10}\}$ are sips. Furthermore, Fig.~\ref{fig:regions}(b) shows the case when the unstable and stable manifolds are associated with two different hyperbolic fixed points.
Thus, computing lobe areas require the knowledge of intersection points as accurately as possible and also identifying them as either pips or sips. 
Specifically, for closed intersecting curves $C_1$ and $C_2$ (see Fig.~\ref{fig:A2minusA1_shaded}) encountered in phase space transport problems, we separate the set of intersection points between the two curves into classes of equivalence. For two curves corresponding to the invariant manifolds of a hyperbolic fixed point, each class of equivalence corresponds exactly to the two {\it pips} and {\it sips} on the segment of the invariant manifold. Then, we can compute the lobe area using the contour integral form of the Eqn.~\eqref{intpoly} which should be close to the lobe area given a well-resolved manifold.

\subsection{Intersection Points}

In this section, we assume that there are only transverse intersections of the curves. A numerical algorithm for efficiently computing the intersection points is presented below. The two curves are closed, so the number of intersection points must be even. We compute the $2 N$ intersection points $p_i$ between the two curves $C_1$ and $C_2$. The unit tangent vector to the curve $C_j$ at point $p_i$ is denoted $\mathbf{1}_j(p_i)$. For each intersection point, $p_i$, we define
\begin{equation}
\rho (p_i) =  \textrm{sgn}(\sin(\theta)) \frac{ ||\mathbf{1}_1 (p_i) \times \mathbf{1}_2 (p_i)||}{\sin(\theta) } = \textrm{sgn}(\sin(\theta)) ||  \mathbf{1}_1 (p_i) || || \mathbf{1}_2 (p_i) || \; .
\label{eqrho}
\end{equation} 
where $\theta$ is the angle between the tangents $\mathbf{1}_1(p_i)$ and $\mathbf{1}_2(p_i)$, as shown in Fig.~\ref{fig:rho_pi_c1_c2}.
\begin{figure}[!ht]
	\centering
	\includegraphics[width=0.45\textwidth]{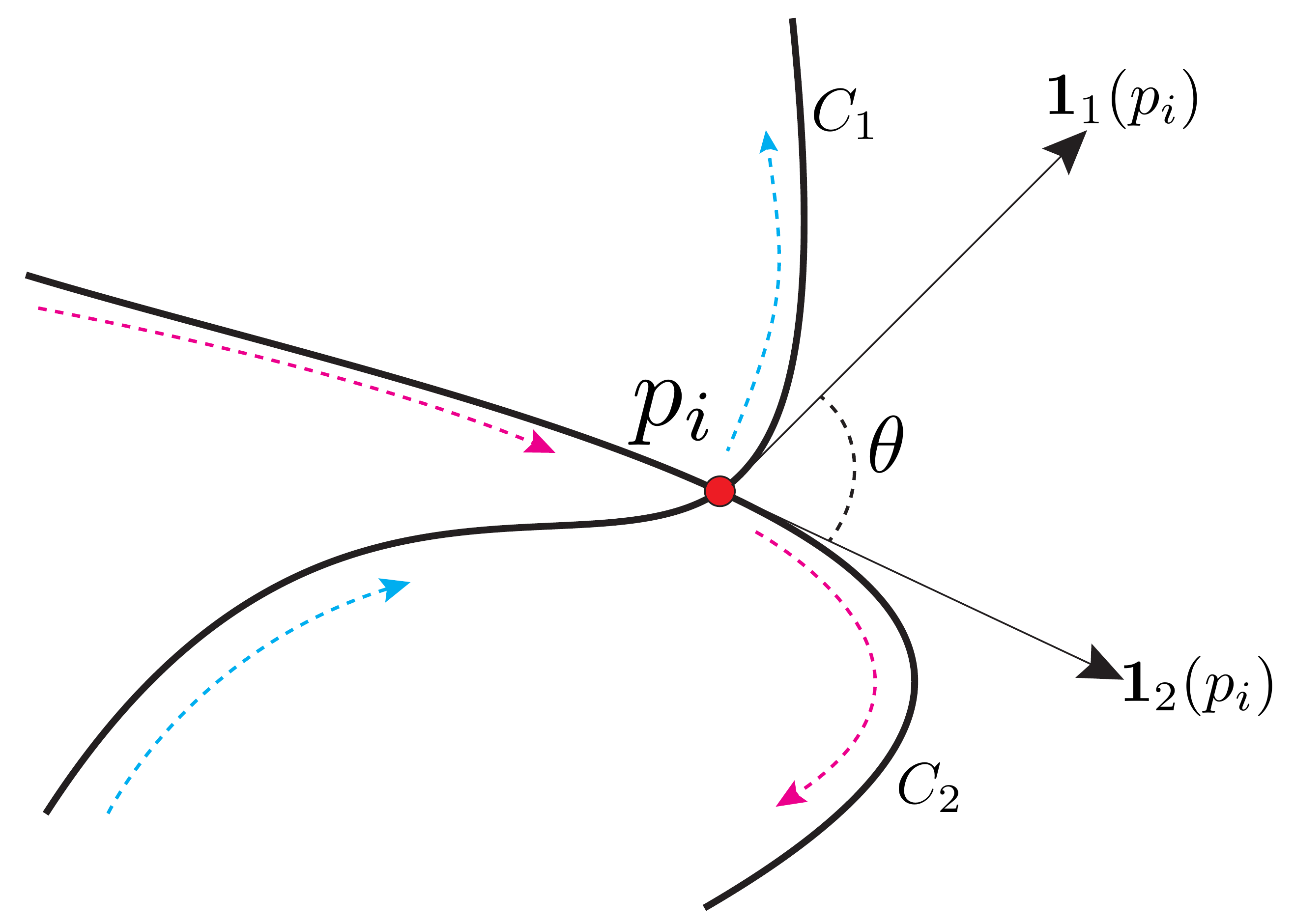}
	\caption{Shows the parameters involved in computing the orientation of an intersection point, $p_i$, when traversing along the curves in the direction shown in dashed arrows.}
	\label{fig:rho_pi_c1_c2}
\end{figure}
The quantity $\rho(p_i)$ represents the \emph{orientation} of a given point $p_i$. We note that when the intersections are transverse, the denominator of Eq.~(\ref{eqrho}) is non-zero and  $\rho (p_i) \in \left\{ -1 , +1\right\}$. Fig.~\ref{lobex1} shows two curves, their intersection points, and the value of $\rho (p_i)$ for each point $p_i$. The closed curves shown in Fig.~\ref{lobex1} can be obtained by extracting the lobes from the intersection of stable and unstable manifolds or intersection of a cylindrical manifold with a plane. It is to be noted that the orientation of the closed curves is related to the geometry of the invariant manifolds, and as such one might have a counter-clockwise and a clockwise oriented curve. 

\begin{figure}[!ht]
	\centering\includegraphics[width=0.95\textwidth]{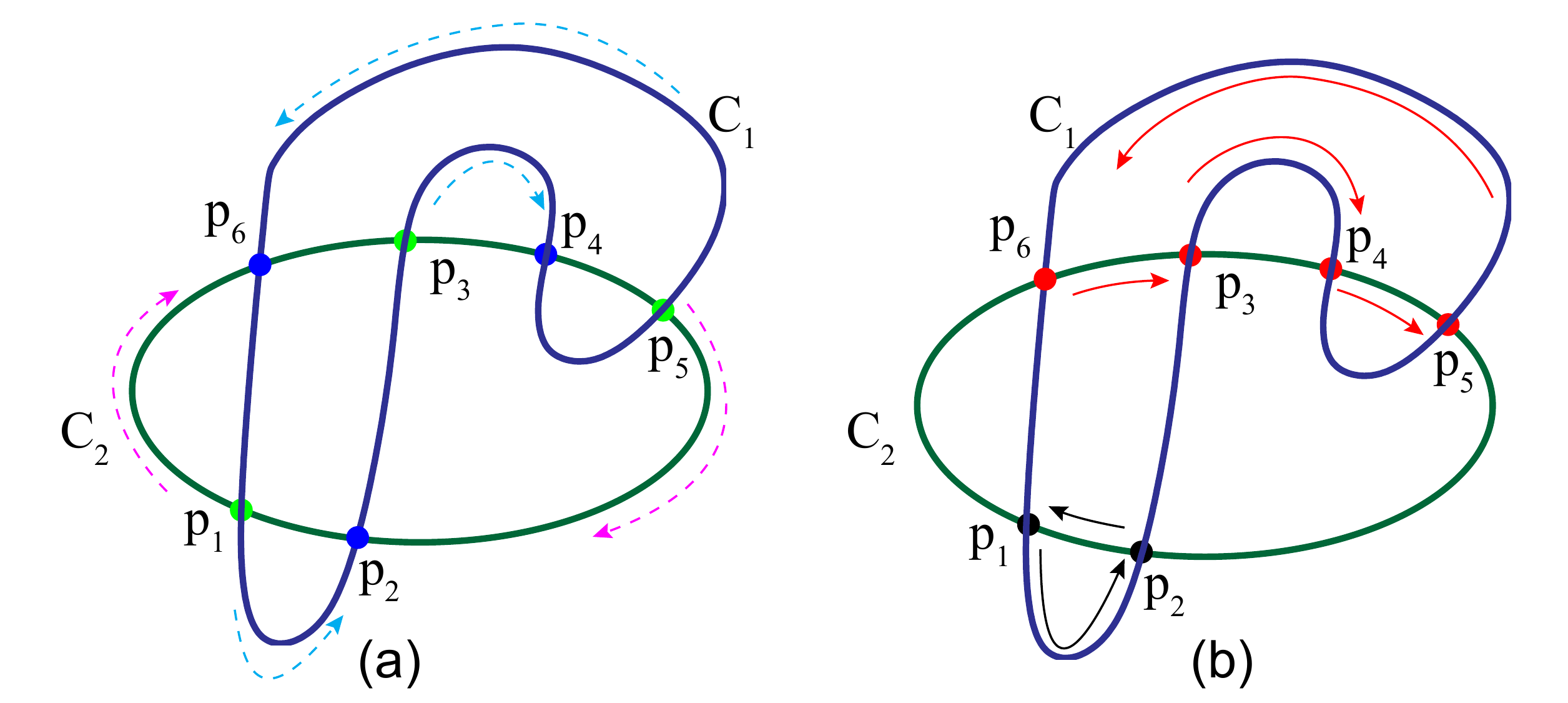} 
	\caption{Example of closed curves oriented in counter-clockwise (cyan arrows) and clockwise direction (magenta arrows) with transverse intersections at $p_1, p_2, \ldots, p_6$. (a) Green intersection points have $\rho (p_i) =1$. Blue intersection points have $\rho (p_i) =-1$. (b) Intersection points of the same color belong to the same equivalency class. The arrows represent the one-to-one and onto relationship $\sigma$ between the intersection points and the points joining the same equivalency class can either be a \textit{knob} (black arrows) or \textit{handle} (red arrows).}
	\label{lobex1}
\end{figure}

\subsection{Classes of Intersection Points}

The segments of curve between intersection points are most important to our computation, so we define $C_1^+[p_i , p_j]$ and $C_2^+ [p_i, p_j]$ as the counter-clockwise segments of respectively $C_1$ and $C_2$ between the points $p_i$ and $p_j$. Similarily, $C_1^-[p_i , p_j]$ and $C_2^- [p_i, p_j]$ are the clockwise segments of respectively $C_1$ and $C_2$ between the points $p_i$ and $p_j$.
We define the positive and negative adjacency on $C_1$ by
\begin{equation}
A^{\pm }_{C_1} (p_i,p_j) = \left\{
\begin{array}{ll}
1&\text{when} \; \; 1\leq k \leq 2 N : k\neq i \text{ and } k\neq j \Longrightarrow p_k \notin C^{\pm }_1[p_i , p_j] \\
0&\text{otherwise}
\end{array}
\right. \; .
\end{equation} and the adjacency on $C_2$ as
\begin{equation}
A^{\pm }_{C_2} (p_i,p_j) = \left\{
\begin{array}{ll}
1&\text{when} \; \; 1\leq k \leq 2 N : k\neq i \text{ and } k\neq j \Longrightarrow p_k \notin C_2^{\pm }[p_i , p_j] \\
0&\text{otherwise}
\end{array}
\right. \; .
\end{equation}

The geometric interpretation of adjacency is that when traversing a curve in counter-clockwise or clockwise (that is, positive or negative sense) direction, the point $p_j$ with adjacency value of 1 is the point next to the point $p_i$. 


The objective of this section is to determine a   generalized notion of a lobe. Intuitively, lobes are bounded by ``segments of the curves $C_1$ and $C_2$ turning in opposite directions on each curve''. We formalize this idea by defining the signed adjacency
\begin{equation}
\gamma (p_i , p_j) = \frac{\rho (p_i) + 1}{2} A^+_{C_1} (p_i,p_j) - \frac{\rho (p_i) - 1}{2} A^-_{C_2} (p_i,p_j) \; .
\label{eqgammadef}
\end{equation} 
which becomes $\gamma (p_i, p_j) = A^+_{C_1} (p_i,p_j)$ for $\rho (p_i) = 1$ and $\gamma (p_i, p_j) = A^-_{C_2} (p_i,p_j)$ for $\rho (p_i) = -1$.
Table~\ref{tab:ac1plus},~\ref{tab:ac2minus}, and~\ref{tab:gamma} show the value of the functions $A^+_{C_1}$, $A^-_{C_2}$, and $\gamma $, respectively, for the example of intersecting curves shown in Fig.~\ref{lobex1} .

\begin{table}[!h]
	\begin{center}
		\begin{tabular}{|c|c|c|c|c|c|c|}\hline
			$A^+_{C_1} (\; \downarrow \; ,\rightarrow ) $&$p_1$&$p_2$&$p_3$&$p_4$&$p_5$&$p_6$\\\hline
			$p_1$&0&1&0&0&0&0\\\hline
			$p_2$&0&0&1&0&0&0\\\hline
			$p_3$&0&0&0&1&0&0\\\hline
			$p_4$&0&0&0&0&1&0\\\hline
			$p_5$&0&0&0&0&0&1\\\hline
			$p_6$&1&0&0&0&0&0\\\hline
		\end{tabular}
	\end{center}
	\caption{Function $A^+_{C_1} $ for the intersection points between the curves of Fig.~\ref{lobex1}. The lines and columns correspond respectively to the first and second argument of $A^+_{C_1} $.}
	\label{tab:ac1plus}
\end{table}

\begin{table}[!h]
	\begin{center}
		\begin{tabular}{|c|c|c|c|c|c|c|}\hline
			$A^-_{C_2} (\; \downarrow \; ,\rightarrow )$&$p_1$&$p_2$&$p_3$&$p_4$&$p_5$&$p_6$\\\hline
			$p_1$&0&0&0&0&0&1\\\hline
			$p_2$&1&0&0&0&0&0\\\hline
			$p_3$&0&0&0&1&0&0\\\hline
			$p_4$&0&0&0&0&1&0\\\hline
			$p_5$&0&1&0&0&0&0\\\hline
			$p_6$&0&0&1&0&0&0\\\hline
		\end{tabular}
	\end{center}
	\caption{Function $A^-_{C_2} $ for the intersection points between the curves of Fig.~\ref{lobex1}. The lines and columns correspond respectively to the first and second argument of $A^-_{C_2} $.}
	\label{tab:ac2minus}
\end{table}

\begin{table}[!h]
	\begin{center}
		\begin{tabular}{|c|c|c|c|c|c|c|}\hline
			$\gamma \; (\; \downarrow \; ,\rightarrow ) $&$p_1$&$p_2$&$p_3$&$p_4$&$p_5$&$p_6$\\\hline
			$p_1$&0&1&0&0&0&0\\\hline
			$p_2$&1&0&0&0&0&0\\\hline
			$p_3$&0&0&0&1&0&0\\\hline
			$p_4$&0&0&0&0&1&0\\\hline
			$p_5$&0&0&0&0&0&1\\\hline
			$p_6$&0&0&1&0&0&0\\\hline
		\end{tabular}
	\end{center}
	\caption{Function $\gamma $ for the intersection points between the curves of Fig.~\ref{lobex1}. The lines and columns correspond respectively to the first and second argument of $\gamma $.}
	\label{tab:gamma}
\end{table}

In order to separate the intersection points in disjoint classes, we define an equivalency relation between the intersection points by
\begin{definition}
	$p_i \sim p_j$ iff there exists a sequence of $K$ intersection points $p_{\alpha _k}$ such that
	\begin{equation}
	\left\{
	\begin{array}{l}
	p_{\alpha _1} = p_i \; ,\\
	p_{\alpha _K} = p_j \; ,\\[3pt]
	\Pi _{k=1}^{K-1} \gamma (p_k , p_{k+1}) = 1 \; .
	\end{array}
	\right.
	\label{eqequrel}
	\end{equation}
\end{definition}

Thus, the sequence of $K$ intersection points bound a domain while traversing the curve in clockwise or counter-clockwise direction.
We have the following properties
\begin{theorem}[Reflexivity of $\sim $]
	\begin{equation}
	\forall i : p_i \sim p_i \; ,
	\end{equation} 
	\label{thsim1}
\end{theorem}
\begin{proof}
	Letting $K=1$ and $\alpha _1 = i$ in Eq.~(\ref{eqequrel}) gives the desired result. $\Box $
\end{proof}

\begin{theorem}[Symmetry of $\sim $]
	\begin{equation}
	\forall i,j : p_i \sim p_j \Longrightarrow p_j \sim p_i \; ,
	\end{equation} 
	\label{thsim2}
\end{theorem}
\begin{proof}
	We notice that for any intersection point $p_i$, there is one and only one intersection point $p_j$ satisfying $A_{C_1}^+(p_i,p_j) = 1$. This is a consequence of the fact that the intersections are transverse and that there are not any self-intersections. There must be at least one segment of the curve $C_1$ leaving $p_i$ (based on a counter-clockwise orientation of $C_1$). If there was more than one segment leaving this point, the curve $C_1$ would self-intersect. Similarly, there is also one and only one intersection point satisfying $A_{C_2}^- (p_i,p_k)=1$. Since $\rho (p_i) \in \left\{ -1,1\right\}$, there is only one intersection point $p_q$ ($q=j$ or $k$) such that $\gamma (p_i , p_q) = 1$. We define $\sigma (p_i)$ as the unique intersection point satisfying
	\begin{equation}
	\gamma (p_i , \sigma (p_i)) = 1 \; .
	\end{equation} Notice that 
	\begin{equation}
	p_i \neq p_j \Longrightarrow \sigma (p_i) \neq \sigma (p_j) \; ,
	\end{equation} and $\sigma $ is therefore invertible.
	We construct the sequence $\left( p_{\alpha _i} \right)$ using
	\begin{equation}
	\left\{ \begin{array}{l}
	p_{\alpha _1} = p_i \; , \\
	p_{\alpha _k+1} = \sigma (p_{\alpha _k}) \; .
	\end{array} \right.
	\end{equation}
	Notice that the sequence $\left( p_{\alpha _i} \right)$ satisfies the third condition of Eq.~(\ref{eqequrel}) for any length. The number of intersection point is finite, so there must be an integer $K$ such that $p_{\alpha _K}$ is identical to $p_{\alpha _k}$ for $k<K$. We cut the sequence at the smallest possible $K$, so there is only one repeated element in the sequence. The repeated element must be $p_{\alpha _1}$ because $\sigma $ is invertible. As a result, the sequence $\left( p_{\alpha _i} \right)$ is the unique path from $p_i$ to $p_i$ that does not contain only $p_i$ and does not contain repeated points except for the endpoints.
	To prove the relationship, we notice that $p_i \sim p_j \Longrightarrow $ the point $p_j$ must be included in the unique sequence $\left( p_{\alpha _i} \right)$. As a result, we can take a subsequence from $p_i$ to $p_j$ in $\left( p_{\alpha _i} \right)$ and $p_j \sim p_i$. $\Box $
\end{proof}

\begin{table}[!h]
	\begin{center}
		\begin{tabular}{|c|c|c|c|}\hline
			&$\rho ( p) $& $\sigma (p)$&Class\\\hline
			$p_1$&$+1$&$p_2$&1\\
			$p_2$&$-1$&$p_1$&1\\
			$p_3$&$+1$&$p_4$&2\\
			$p_4$&$-1$&$p_5$&2\\
			$p_5$&$+1$&$p_6$&2\\
			$p_6$&$-1$&$p_3$&2\\\hline
		\end{tabular}
	\end{center}
	\caption{Functions $\rho $ and $\sigma $ for the intersection points between the curves of Fig.~\ref{lobex1}. There are two equivalency classes for this example.}
	\label{tab:rhosigmaclass}
\end{table}

\begin{theorem}[Transitivity of $\sim $]
	\begin{equation}
	\forall i,j,k : \left\{ 
	\begin{array}{l}
	p_i \sim p_j \\
	p_j \sim p_k
	\end{array}
	\right. \Longrightarrow p_i \sim p_k \; .
	\end{equation} 	
	\label{thsim3}
\end{theorem}
\begin{proof}
	The hypothesis gives us one path from $p_i $ to $p_j$ and one path from $p_j $ to $p_k$. We can combine these two paths to go from $p_i$ to $p_k$ and $p_i \sim p_k$. $\Box $
\end{proof}

As a result, $\sim $ is an equivalence relation for the set $P$ of intersection points. 

\subsection{Lobe Area}
Thus, the lobe area can be computed by using the contour integral form Eqn.~\eqref{intpoly} for the intersection points that are equivalent under $\sim$. Let us define
\begin{equation}
S(p) = \left\{ \begin{array}{ll} 
\int _{C_1^+[p,\sigma (p)]} y dx - x dy &\text{if} \; \rho (p) = +1 \; , \\
\int _{C_2^-[p,\sigma (p)]} y dx - x dy &\text{if} \; \rho (p) =  -1\; ,
\end{array}
\right.
\label{eqiparea1}
\end{equation}
where $\sigma (p)$ has been defined in the proof of Theorem~\ref{thsim2}. Then, the lobe area can be computed using the following 

\begin{conjecture}[Lobe area]
	\begin{equation}
	[ A_1 \backslash A_2 ] = \sum _{P_i \in P\backslash \sim} \left( \sum _{p \in P_i} S(p) \right) \; .
	\label{eqiparea2}
	\end{equation}
	\label{thareaip}
\end{conjecture}
where we define the classes of equivalency $\left\{ P_i \right\} = P \backslash\sim$ as corresponding to subsets of $P$ containing all the elements that are equivalent under $\sim $ (i.e., the quotient of $P$ by $\sim $). We define the number of sets in $P \backslash \sim$ as $n_P$ and we note that the equivalency classes constitute a partition of $P$. Thus, the $P_i$ are disjoint and the union of their elements is $P$. For the example intersecting curves in Fig.~\ref{lobex1}, there are $n_P=2$ equivalency classes, $P_1=\{p_1,p_2\}$ and $P_2=\{p_3,p_4,p_5,p_6\}$.

It is to be noted that using $A^-_{C_1}$ and $A^+_{C_2}$ instead of $A^+_{C_1}$ and $A^-_{C_2}$ in the definition of $\gamma $ (Eq.~\ref{eqgammadef}) gives another equivalence relationship. The equivalence classes of the latter gives the area $[A_2\backslash A_1]$ in a form identical to Proposition~\ref{thareaip}. We also note that Proposition~\ref{thareaip} could be written with a single sum over all the intersection points. However, we prefer to keep the sum of the equivalency classes because the sum over each intersection point in the same class correspond to the area of a lobe, that is, an individual protuberance of $A_1$ outside $A_2$. 

\subsection{Numerical Implementation}

\subsubsection{Intersection Points}

The simplest approach of computing intersection points between two curves that are piecewise linear requires browsing every pair of linear segments (one on $C_1$ and one on $C_2$), and determining the possible intersection point between these two segments. This can become a computationally expensive operation for a manifold with million points since this approach requires solving a 2 $\times $ 2 linear system at each of the pair of piecewise linear segments, and hence requiring $O(N^2)$ time, where $N$ is number of line segments. Furthermore, the Bentley–Ottmann algorithm uses a line to sweep across $N$ line segments with $K$ intersection points in $O(N+K) \log N$ time, and randomized algorithms can solve the problem in $O(N \log N + K)$ time. However, the application of these computational geometry algorithms is much more complicated, and in case of finite precision arithmetic leads to other challenges~\cite{Shamos1976,orourke1998computational}. Instead, we propose necessary and sufficient conditions to quickly check the existence of an intersection point between two piecewise linear segments. 

\begin{theorem}
	If the segments $[(x_1,y_1),(x_2,y_2)]$ and $[(x'_1,y'_1),(x'_2,y'_2)]$ intersect, then we must have
	\begin{equation}
	\left( (y_2-y_1) x'_1 - (x_2-x_1) y'_1 - x_1 y_2 + y_1 x_2 \right) \left( (y_2-y_1) x'_2 - (x_2-x_1) y'_2 - x_1 y_2 + y_1 x_2 \right) \leq 0 \; ,
	\end{equation}
	\label{thcnip}
\end{theorem}
\begin{proof}
	The equation of the line passing through the two points $(x_1,y_1)$ and $(x_2,y_2)$ is
	\begin{equation}
	f(x,y)=(y_2-y_1) x - (x_2-x_1) y - x_1 y_2 + y_1 x_2 = 0 \; .
	\end{equation} The theorem states that if the segment $[(x'_1,y'_1),(x'_2,y'_2)]$ intersects the segment $[(x_1,y_1),(x_2,y_2)]$, then it must also intersect the line $f(x,y)=0$. As a result, the endpoints of the segment $[(x'_1,y'_1),(x'_2,y'_2)]$ must be on opposite sides of the line and we have
	\begin{equation}
	f(x'_1,y'_1) f(x'_2,y'_2) \leq 0 \; .
	\end{equation} $\Box $
\end{proof}

Theorem~\ref{thcnip} gives us a necessary condition for an intersection between two segments. Notice that we can reverse the role of each segment in Theorem~\ref{thcnip} and get another necessary condition. In addition, we have

\begin{theorem}
	The segments $[(x_1,y_1),(x_2,y_2)]$ and $[(x'_1,y'_1),(x'_2,y'_2)]$ intersect if and only if
	\begin{equation}
	\left( (y_2-y_1) x'_1 - (x_2-x_1) y'_1 - x_1 y_2 + y_1 x_2 \right) \left( (y_2-y_1) x'_2 - (x_2-x_1) y'_2 - x_1 y_2 + y_1 x_2 \right) \leq 0 \; ,
	\end{equation} and
	\begin{equation}
	\left( (y'_2-y'_1) x_1 - (x'_2-x'_1) y_1 - x'_1 y'_2 + y'_1 x'_2 \right) \left( (y'_2-y'_1) x_2 - (x'_2-x'_1) y_2 - x'_1 y'_2 + y'_1 x'_2 \right) \leq 0 \; .
	\end{equation}
	\label{thcnsip}
\end{theorem}
\begin{proof}
	Theorem~\ref{thcnip} directly implies one direction ($\Longrightarrow $) of the equivalence. To prove $\Longleftarrow $, notice that if the two equations are satisfied, the endpoints of the first segment are on each side of the line containing the second segment. The endpoints of the second segment are also on each side of the line containing the first segment. As a result, the two segments must intersect in at least one point. $\Box $
\end{proof}

The two theorems above allow for a very fast and efficient algorithm to detect intersection points. Each segment on curve $C_1$ is checked for intersections with each segment on curve $C_2$. However, only the necessary condition given by Theorem~\ref{thcnip} is checked. Only if this condition is satisfied is the second condition in Theorem~\ref{thcnsip} checked. If the necessary and sufficient condition is satisfied, then the intersection point is effectively computed.

\subsubsection{Area bounded by a closed curve} 

We assume that the curves $C_1$ and $C_2$ are given in terms of a sequence of points $(x_i,y_i)$. We want an exact evaluation of the integral in Eq.~\eqref{eqiparea1} for piecewise linear curves. By using Eq.~(\ref{intpoly}) in the section above, Eq.~(\ref{eqiparea1}) becomes 
\begin{equation}
S(p) = \sum _{\begin{array}{c} 
	(x_i,y_i) \\ 
	\in \mathcal {C}^{\rho (p)} [p,\sigma (p)]
	\end{array} }\left( y_i x_{i+1} - x_i y_{i+1} \right) \; .
\label{eqiparea3}
\end{equation}
where $\mathcal{C} = C_1/C_2$ and in a form suitable for numerical implementation.  


\section{Non-Transverse Intersections}\label{sect:non-trans-inter}

In this section, in contrast with~\S~\ref{sect:inter-pts-lobes}, we {\em relax} the assumption that the two curves can only have transverse intersections. Thus, we allow non-transverse intersections of the curves, and the input curves can have common segments as shown in Fig.~\ref{fig:closed_curves}. This scenario arises in applying lobe dynamics on the Poincar\'e surface-of-section in 2 or more degrees of freedom, time-periodic 1-DOF system, 2 dimensional time-periodic fluid flow, and when multilobe, self-intersecting turnstiles are formed. This is shown in Fig.~\ref{fig:multi-lobe-turnstile}, and discussed in detail in Refs.~\cite{Koon2004,Dellnitz2005,Ross2012}. In these case of multilobe turnstiles, transverse intersection cannot be guaranteed.
\begin{figure}[!ht]
	\centering
	\includegraphics[width=0.95\textwidth]{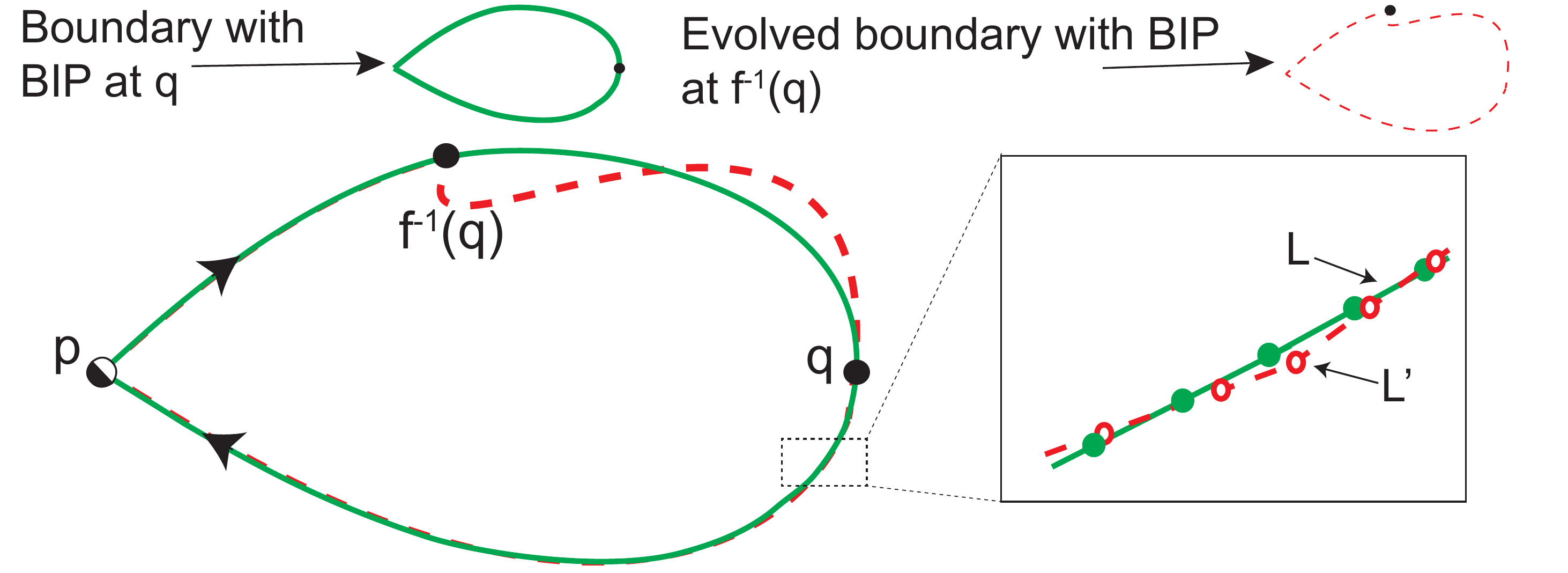}
	\caption{Schematic showing the boundary (green solid curve) and its pre-image (red dashed curve) with boundary intersection point (BIP) at $q$ and $f^{-1}(q)$, respectively. When the boundaries have near-tangent intersections, computing intersection points requires expanding/shrinking the points on the curve to avoid false lobes formed by the segments of curve $L$ and $L'$.}
	\label{fig:closed_curves}
\end{figure}
\begin{figure}[!ht]
	\centering
	\subfigure[]{\includegraphics[width=0.45\textwidth]{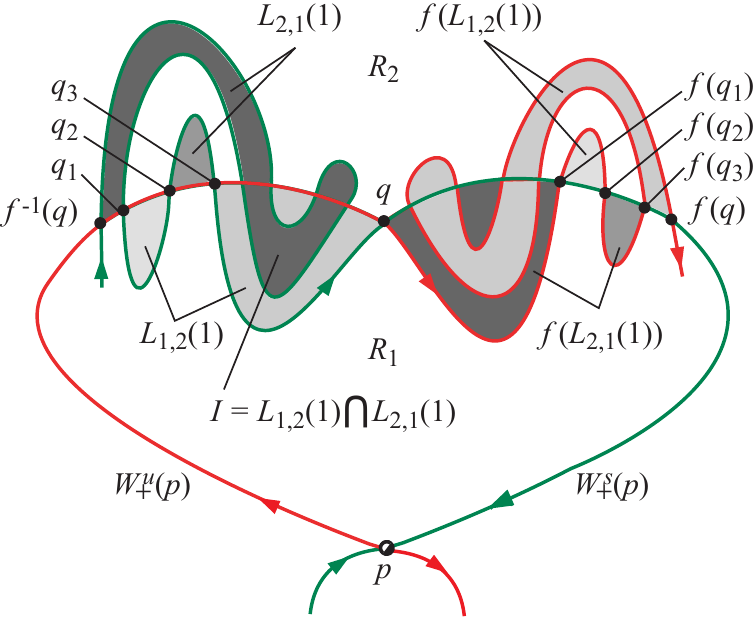}}
	\subfigure[]{\includegraphics[width=0.45\textwidth]{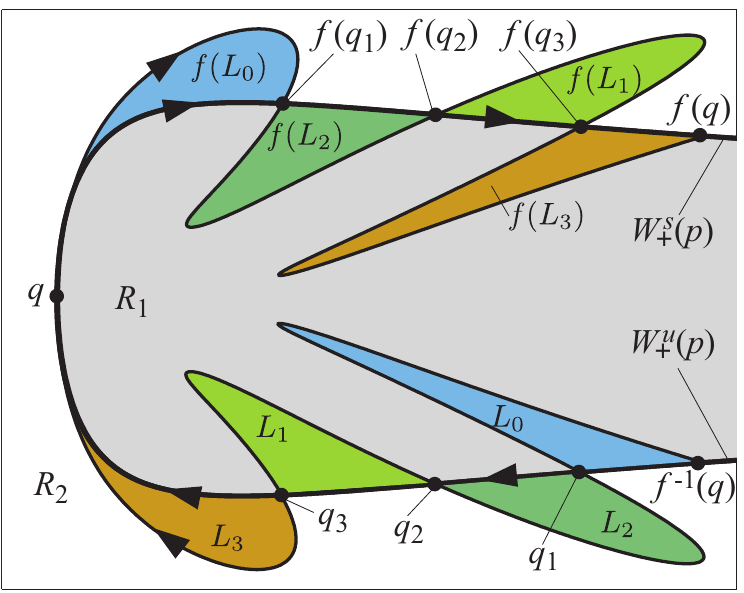}}
	\caption{(a) Schematic showing the geometry of multilobe, self-intersecting lobes. (b) An example of a multilobe turnstile that arises when considering the motion of a small body (for example, ejecta) in the field of a rotating asteroid; see Ref.~\cite{Koon2004} for details.}
	\label{fig:multi-lobe-turnstile}
\end{figure}
In the case of transverse intersections, we used an algorithm based on primary intersection points and secondary intersection points to identify the boundary of each lobe, which was then used to compute the area enclosed by the boundary.  The near-tangency (when the determinant of the linear system resulting from the pair of linear segments is close to 0) of intersection of the two curves creates both a confusion in the ordering of the intersection points and the position of the lobes as well as an enormous computational difficulty to extract the actual position of the intersection. Thus, to get around the non-transverse intersection, we resort to computing a function that is amenable to the tangle geometry of manifolds and lobes.


\subsection{Interior Function}

Let us define the complex function 
\begin{equation}
f_{z_0}(z) = \frac{1}{x-x_0+i(y-y_0)}, \quad \text{where} \quad z = x + iy 
\end{equation} The integral of $f(z)$ over a closed curve is equal to $2i\pi $ times the number of turns that the closed curve makes around the point $x_0 + i y_0$ (to prove that, use the fact that $f(z)$ is analytic everywhere in the complex plane except at $x_0 + i y_0$ and use the residue theorem).
We define
\begin{equation}
J_i (x_0, y_0) = Im \left\{ \int _{C_i} \frac{dx+i dy}{x-x_0+i(y-y_0)} \right\} ,
\label{eqJi}
\end{equation} and
\begin{equation}
I_i (x_0, y_0) =\left\{
\begin{array}{cl}
1 & \text{if } J_i (x_0, y_0) < \pi\\
-1 & \text{if } J_i (x_0, y_0) \ge \pi
\end{array}
\right. ,
\end{equation}
We note that $I_i (x_0, y_0) $ is negative when the point $(x_0, y_0)$ is contained in the simple closed curve $A_i$, and is positive otherwise.

\subsection{Lobe Area}

In order to extract $[A_1 \backslash A_2]$ and $[A_2 \backslash A_1]$ from the shape of the curves, we define the following quantities
\begin{equation}
Q_1 = \int _{C_1} I_2(x,y) \left( y dx - xdy \right) + \int _{C_2} I_1(x,y) \left( y dx - xdy \right),
\label{eqq1}
\end{equation}

\begin{equation}
Q_2 = \int _{C_1} \frac{I_2(x,y)+1}{2} \left( y dx - xdy \right) + \int _{C_2} \frac{I_1(x,y)+1}{2} \left( y dx - xdy \right),
\label{eqq2}
\end{equation}

\begin{equation}
Q_3 = \int _{C_1} \frac{I_2(x,y)-1}{-2} \left( y dx - xdy \right) + \int _{C_2} \frac{I_1(x,y)-1}{-2} \left( y dx - xdy \right).
\label{eqq3}
\end{equation} 
A quick look at the different paths on Fig.~\ref{figpathqs} reveals that
\begin{figure}
	\centering\includegraphics[width=5in]{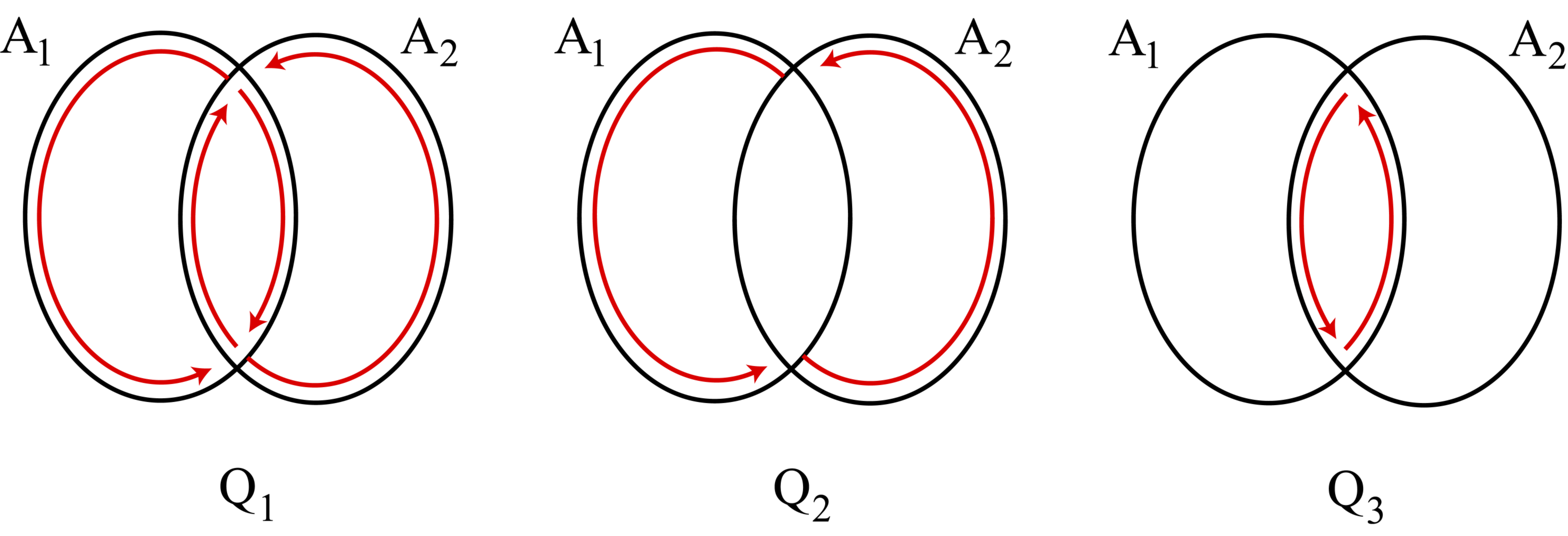}
	\caption{Schematic view of the paths involved in the definition of $Q_1$, $Q_2$ and $Q_3$.}
	\label{figpathqs}
\end{figure}
\begin{equation}
Q_1=[A_1\cup A_2] - [A_1\cap A_2] = [A_1\backslash A_2] +[A_2\backslash A_1]
\end{equation}
\begin{equation}
Q_2=[A_1\cup A_2] ,
\end{equation}
\begin{equation}
Q_3=[A_1\cap A_2] .
\end{equation}
Notice that the three equation above are not linearly independent and {\em any} of them can give us the expected results. However, we compute $Q_1$, $Q_2$ and $Q_3$ and use the redundancy to minimize the error on the integrals and provide an approximation of the computational error. Since
\begin{equation}
[A_1 \backslash A_2] = [A_1 \cup A_2] - [A_2] = [A_1] - [A_1\cap A_2] ,
\end{equation} and
\begin{equation}
[A_2 \backslash A_1] = [A_1 \cup A_2] - [A_1] = [A_2] - [A_1\cap A_2], 
\end{equation} we have
\begin{equation}
[A_1 \backslash A_2] = \frac{1}{2} (Q_2-Q_3) + \frac{1}{2} ([A_1]-[A_2]),
\end{equation} and
\begin{equation}
[A_2 \backslash A_1] = \frac{1}{2} (Q_2-Q_3) + \frac{1}{2} ([A_2]-[A_1]).
\end{equation}
And using
\begin{equation}
[A_1]+[A_2]=[A_1 \cup A_2]+[A_1 \cap A_2],
\end{equation} we find
\begin{equation}
[A_1 \backslash A_2] = \frac{1}{2} Q_1 + \frac{1}{2} ([A_1]-[A_2]),
\end{equation} and
\begin{equation}
[A_2 \backslash A_1] = \frac{1}{2} Q_2 + \frac{1}{2} ([A_2]-[A_1]).
\end{equation}
Our algorithm combines these two results and provides the following final answer
\begin{equation}
[A_1 \backslash A_2]  =  \frac{1}{2} (A_1 - A_2) + \frac{1}{4} (Q_1+Q_2-Q_3) ,\end{equation}
\begin{equation}
[A_2 \backslash A_1]  =  \frac{1}{2} (A_2 - A_1) + \frac{1}{4} (Q_1+Q_2-Q_3) ,
\end{equation}
\begin{equation}
\delta [A_1 \backslash A_2]  =  \delta [A_2 \backslash A_1 ] = \frac{1}{4} \left| Q_2-Q_3-Q_1 \right| \; ,
\end{equation}
where the last equation gives the approximate error on the computed area.

\subsection{Numerical Methods}

The only technical difficulty in this algorithm is the numerical computation of the functions $J_i(x_0,y_0)$ in Eq.~\eqref{eqJi}, and we present an algorithm for piecewise linear boundaries in \S~\ref{jicomp}. The computation of the integrals in Eq.~\eqref{eqq1}),~\eqref{eqq2}, and~\eqref{eqq3} is similar to the numerical method used for the integral in \S~\ref{contourint}. In addition, we present an algorithm to increase the number of points on the input curve close to the intersection points; we refer to the implementation of this algorithm in \emph{Lober} as \emph{densifier}, and its use is presented in \S~\ref{subsubsect:densifier}. The algorithm in \S~\ref{subsubsect:densifier} has proved to increase the accuracy of this method without compromising the computational cost too much.

\subsubsection{Computation of $J_i(x,y)$ }\label{jicomp}
In this section, we derive the exact value of the function $J_i$ defined in Eq.~(\ref{eqJi}) when the boundary of $C_i$ is given as a sequence of piecewise linear segments. For each linear segment $[(x_1,y_1),(x_2,y_2)]$, we have
\begin{equation}
\left\{ \begin{array}{lcl}
x=x_1 + t (x_2 - x_1) \; ,\\
y=y_1 + t (y_2 - y_1) \; .
\end{array} \right.
\end{equation} 
Using $dx = (x_2 - x_1)  \; dt$ and $dy = (y_2 - y_1) \; dt$, we obtain
\begin{align}
Im \left\{ \frac{dx + i dy}{x-x_0 + i (y-y_0)} \right\} & = \frac{(x_1-x_0)(y_2-y_1) \; dt - (y_1 - y_0)(x_2-x_1) \; dt}{\left(x_1-x_0+t(x_2-x_1)\right)^2+\left(y_1-y_0+t(y_2-y_1)\right)^2}\; , \\
& = \bar{\mathbf{1}}_z \frac{(\bar{x}_1-\bar{x}_0)\times (\bar{x}_2-\bar{x}_1) \; dt}{\left\| \bar{x}_2 - \bar{x}_1 \right\| ^2 t^2+ 2 (\bar{x}_2 - \bar{x}_1) \cdot (\bar{x_1} - \bar{x}_0) t+ \left\| \bar{x}_1 - \bar{x}_0 \right\| ^2} \; ,
\end{align}
where $\bar{\mathbf{1}}_z$ is the unit normal vector in the z-direction and $\bar{x}_0 = (x_0, y_0)$, $\bar{x}_1 = (x_1, y_1)$, $\bar{x}_2 = (x_2, y_2)$.
\begin{align}
Im \left\{ \frac{dx + i dy}{x-x_0 + i (y-y_0)} \right\} & = \frac{\left\| \bar{u} \right \| ^ 2 \left\| \bar{v} \right\| ^2 \sin (\bar{u},\bar{v}) \; dt}{\left\| \bar{v} \right\| ^2 t^2 + 2 \left\| \bar{u} \right\| \left\| \bar{v} \right\| \cos (\bar{u},\bar{v}) \; t + \left\| \bar{u} \right\| ^2} \; ,
\label{eqJip1}
\end{align} 
where
\begin{equation}
\left\{ \begin{array}{lcl}
\bar{u} & = & \bar{x}_1 - \bar{x}_0 \; , \\
\bar{v} & = & \bar{x}_2 - \bar{x}_1 \; .
\end{array} \right.
\end{equation}
We note that the discriminant of the denominator of Eq.~\eqref{eqJip1} is
\begin{equation}
\rho = - 4 \left\| \bar{u} \right \|^2 \left\| \bar{v} \right\| ^2 \sin ^2 (\bar{u},\bar{v}) \leq 0 \; .
\end{equation}
As a result, we have
\begin{eqnarray}
Im \left\{ \int _{[\bar{x}_1,\bar{x}_2]} \frac{dx + i dy }{x-x_0 + i (y-y_0)} \right\} & = & \int _0 ^ 1 \frac{ \left\| \bar{u} \right\| ^ 2 \left\| \bar{v} \right\| ^ 2 \sin (\bar{u} , \bar{v}) \;  dt}{\left\| \bar{v} \right\| ^2 t^2 + \left\| \bar{u} \right\| \left\| \bar{v} \right\| \cos (\bar{u}, \bar{v}) + \left\| \bar{u} \right\| ^2 }\\
& = & \frac{ \sin ( \bar{u} , \bar{v} )}{\left| \sin ( \bar{u} , \bar{v}) \right| } \left[ \tan ^{-1}  \left( \frac{\left\| v \right\| ^ 2 t + \bar{u} \cdot \bar{v} }{ \left\| \bar{u} \right\| \left\| \bar{v} \right\| \left| \sin (\bar{u} , \bar{v} ) \right|} \right) \right] _0^1 \\
& = & \tan ^{-1} \left( \frac{\left\| \bar{v} \right\| ^2 + \bar{u}\cdot \bar{v}}{\bar{\mathbf{1}}_z \cdot (\bar{u} \times \bar{v})} \right) -\tan ^{-1}  \left(  \frac{\bar{u}\cdot \bar{v}}{\bar{\mathbf{1}}_z \cdot (\bar{u} \times \bar{v})} \right)
\end{eqnarray}
We note that for $\sin (\bar{u},\bar{v}) = 0$, the increment to the integral is zero. This is consistent with the equation above where the right-hand term is continuous for $\bar{u}\times \bar{v} = 0$ and vanishes. 
This approach is similar to the point-in-polygon problem that is tackled in computational geometry (see Ref.~\cite{orourke1998computational}) by using the ray sweep method to check if the point is inside or outside a polygon. 
winding number approach. Although, our method is mathematically satisfying, it may show poor performance when compared to the more efficient method of ray passing. This is due to the integral and trigonometric function evaluations that are involved, but our approach is to adopt a method amenable to manifolds as input curves at the expense of computational performance. 

%

\subsubsection{Curve Densifier}\label{subsubsect:densifier}

There is a built-in \emph{densifier} module in the light version of the package \emph{Lober}, and which adds points on the curves close to the intersection points. The densifier can be activated by adding the parameters \verb+-DENS <nPass> <nDens>+ at the command line. The arguments \verb+<nPass>+ and \verb+<nDens>+ give the number of passes to be performed and the number of points to add near each intersection at each pass, respectively.

The extra precision is always $i_r = {n_{dens}}^{n_{pass}}$, where $n_{dens}$ = \verb+<nDens>+ and $n_{pass}$ = \verb+<nPass>+. In other words, the precision is the same with $(n_{pass}=1 , n_{dens}=1000)$ than with $(n_{pass}=3, n_{dens}=10)$. For a constant $i_r$, the value of the two parameters $n_{dens}$ and $n_{pass}$ should minimize the computational time. Small $n_{pass}$ means that fewer steps are necessary to densify the curve and can reduce the computational time. However, small $n_{pass}$ usually implies large $n_{dens}$ to maintain a constant $i_r$. Since the extra length of the curve is $n_{dens} n_{pass}$, the number of points increases rapidly if $n_{pass}$ is too small and lengthens the computation. So there is an optimal $n_{pass}$ that minimizes computation time, and this optimal $n_{pass}$ parameter depends on the local curvature of the manifold under consideration.


We present a comparison of the \emph{densifier} module in Fig.~\ref{fig:dens_compare} by using this to compute the lobe area for two application problems. As demonstrated by the results, the accuracy in the lobe area does not improve dramatically beyond addition of $10^3$ points, and is comparable to the default densifier in the light version of \emph{Lober}.
\begin{figure}[!ht]
	\centering
	\subfigure[OVP flow]{\includegraphics[scale=0.25]{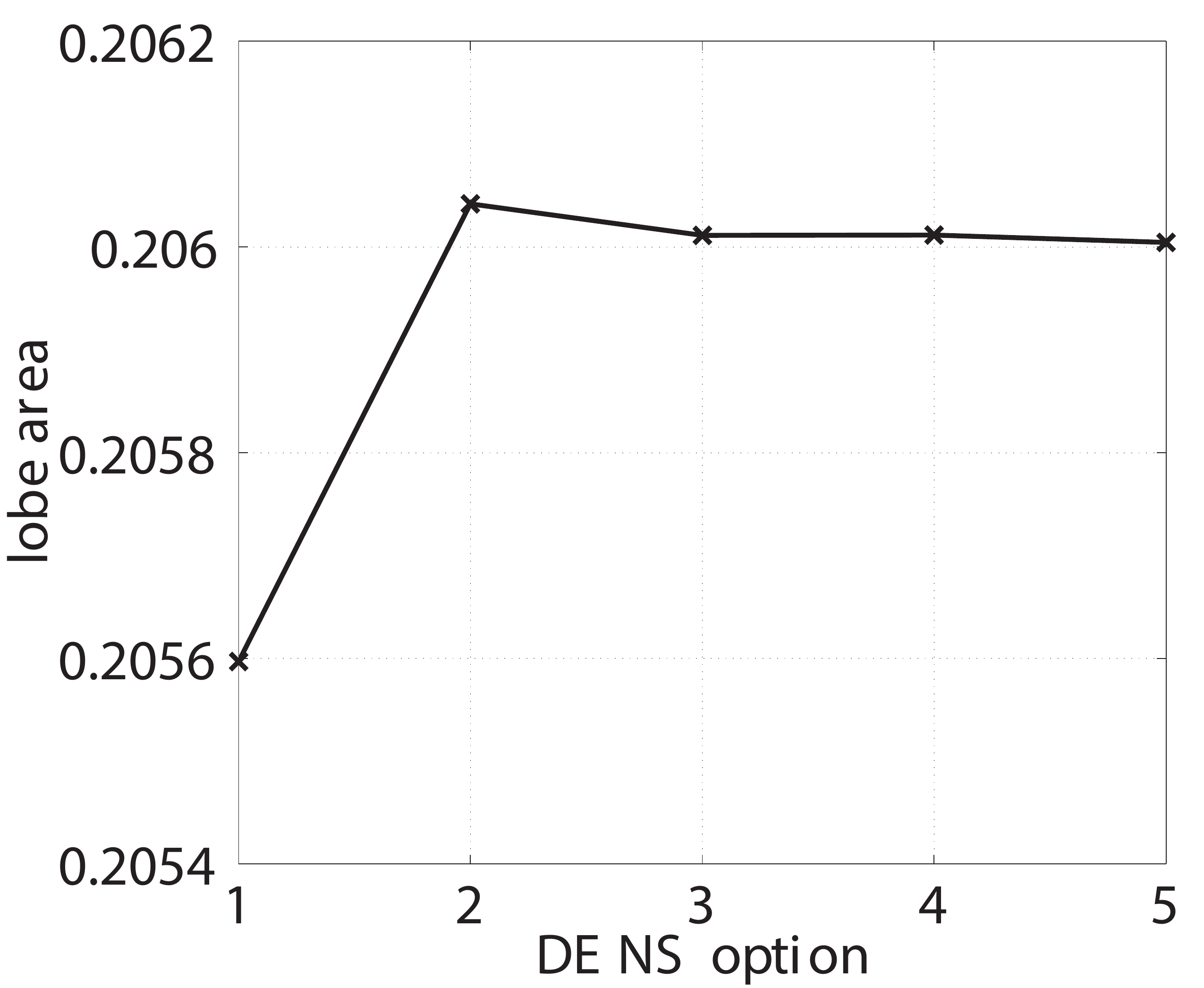}}
	\subfigure[PCR3BP]{\includegraphics[scale=0.27]{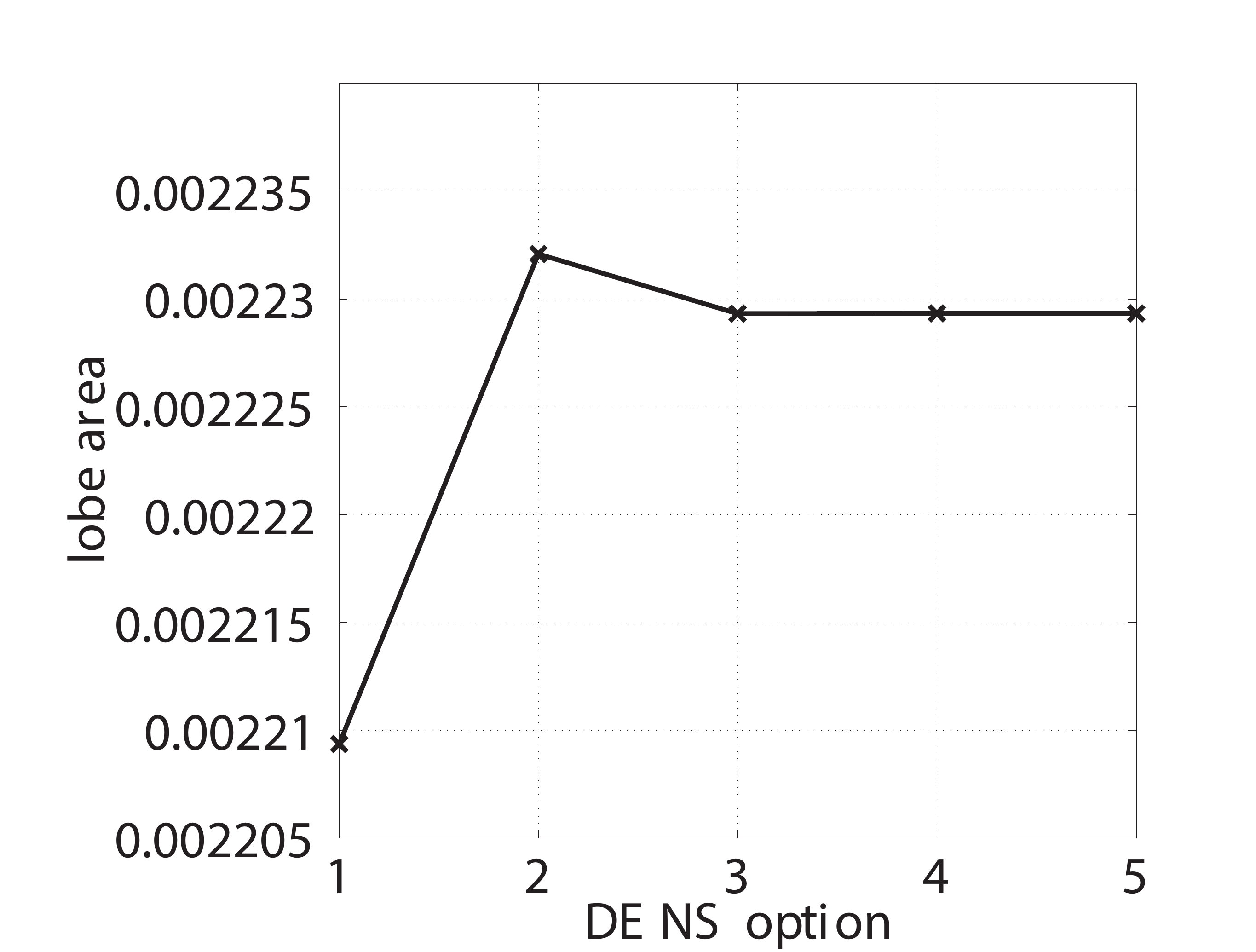}}
	\caption{Shows the numerically computed lobe area that used \emph{densifier} option. (a) oscillating vortex pair (OVP) flow as discussed in Ref.~\cite{Rom-kedar1990}, and (b) planar circular restricted three body problem (PCR3BP) as discussed in Ref.~\cite{Koon2011a}. Along the x-axis, the indices denote the following \emph{densifier} options: 1. \texttt{-DENS 0}, 2. \texttt{-DENS 1 10}, 3. \texttt{-DENS 3 10}, 4. \texttt{-DENS 5 100}, 5. without -DENS option which triggers the default densifier. The y-axis is the average of the entraining and detraining lobes for the parameters used in the references.} 
	\label{fig:dens_compare}
\end{figure}

\section{Application to lobe dynamics and rate of escape}\label{sect:applications} 
In this section, we present applications of the computational approach to problems of chaotic transport in phase space that uses lobe dynamics in fluid flow and rate of escape from a potential well in a ship capsize model. In essence, using the numerical methods presented in \S~\ref{sect:inter-pts-lobes} and \S~\ref{sect:non-trans-inter}, we demonstrate the qualitative and quantitative results that is relevant for these geometric methods of phase space transport.

%




\subsection{Chaotic transport in fluid flow}
In this section, we apply the numerical methods to quantify phase space transport in the oscillating vortex pair (OVP) flow that was introduced in \citeauthor{Rom-kedar1990} \cite{Rom-kedar1990}. In this example from fluid dynamics, it is of interest to calculate the transport rate, in terms of the iterates of a two dimensional map, of species $S_1$ from the inner core $R_1$ to the outside region $R_2$ shown in  Fig.~\ref{fig:ovp_unperturbed}(b). More generally and referring to the Fig.~\ref{fig:regions}(b), this transport in phase space can be discussed using the language of lobe dynamics and transport in two dimensional maps developed in Refs.~\cite{Rom-kedar1990,Wiggins1992chaotic}. We will briefly summarize the OVP flow and the transport quantities that are relevant for the numerical experiments in our study.  

In two-dimensional, incompressible, inviscid fluid flow, the stream-function is analogous to the Hamiltonian governing particle motion and the domain of the fluid flow is identified as the phase space. The OVP flow is generated by two counterclockwise rotating vortices under sinusoidal perturbation which can generate chaotic particle trajectories that causes regions of phase space to move across transport barriers. That is, time dependent velocity field, that is referred to as unsteady flow in fluid dynamics, can generate particle motion between the region with closed streamlines (shown as magenta in Fig.~\ref{fig:ovp_unperturbed}(a)) and the region outside the heteroclinic connection (shown as black in Fig.~\ref{fig:ovp_unperturbed}(a)). When the flow is time-dependent and has periodic forcing, a suitable approach is to construct a two-dimensional map of the flow by defining Poincar\'e section at a certain phase of the flow and given by
\begin{equation}
\begin{aligned}
f :& U^0 \rightarrow U^0  \\
\text{where,} \qquad U^0 &= \{(x,y,\theta) \in \mathbb{R}^2 \times \mathbb{S}^1 : \theta = 0 \}
\end{aligned}
\label{eqn:flow_map2D}
\end{equation}
Thus, when the fluid flow is time-dependent and has a sinusoidal perturbation, the hyperbolic fixed points of the two dimensional map and the associated invariant manifolds can be used to define a boundary parametrized by an intersection point, which acts as a \emph{partial barrier} between regions of the phase space. We note here that partial barriers are the analogous to \emph{separatrix} which are complete barriers of transport in the unperturbed fluid flow. When perturbation is applied to fluid flow, the partial barrier, however, acts as a boundary, shown as black line connecting $p_-, q, p_+$ in Fig.~\ref{fig:ovp_unperturbed}(b), across which transport occurs via specific regions of phase space enclosed by the invariant manifolds, lobes. 
\begin{figure}[!ht]
	\centering             
	\subfigure[]{\includegraphics[width=0.375\linewidth]{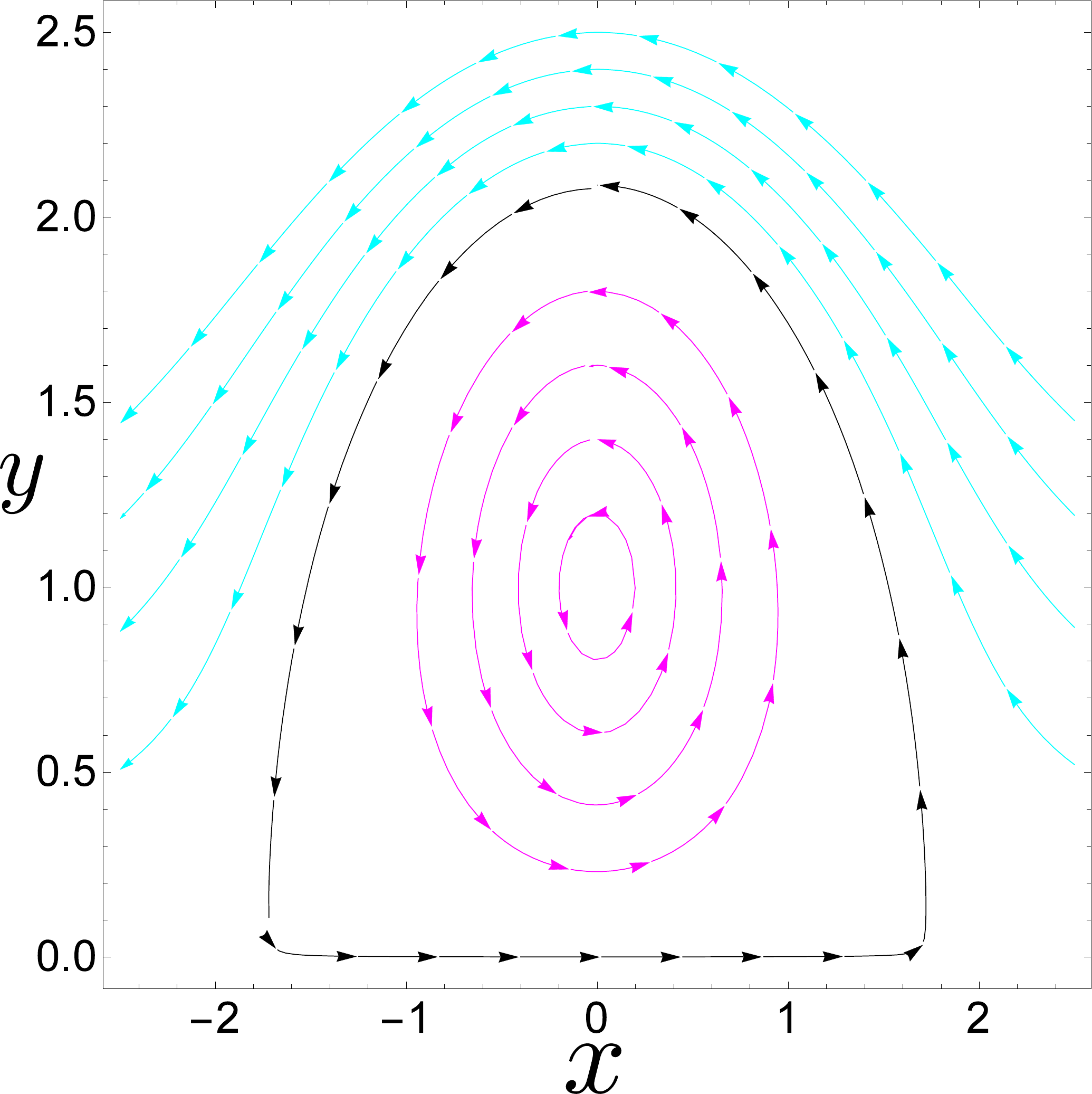}} 
	\subfigure[]{\includegraphics[width=0.475\linewidth]{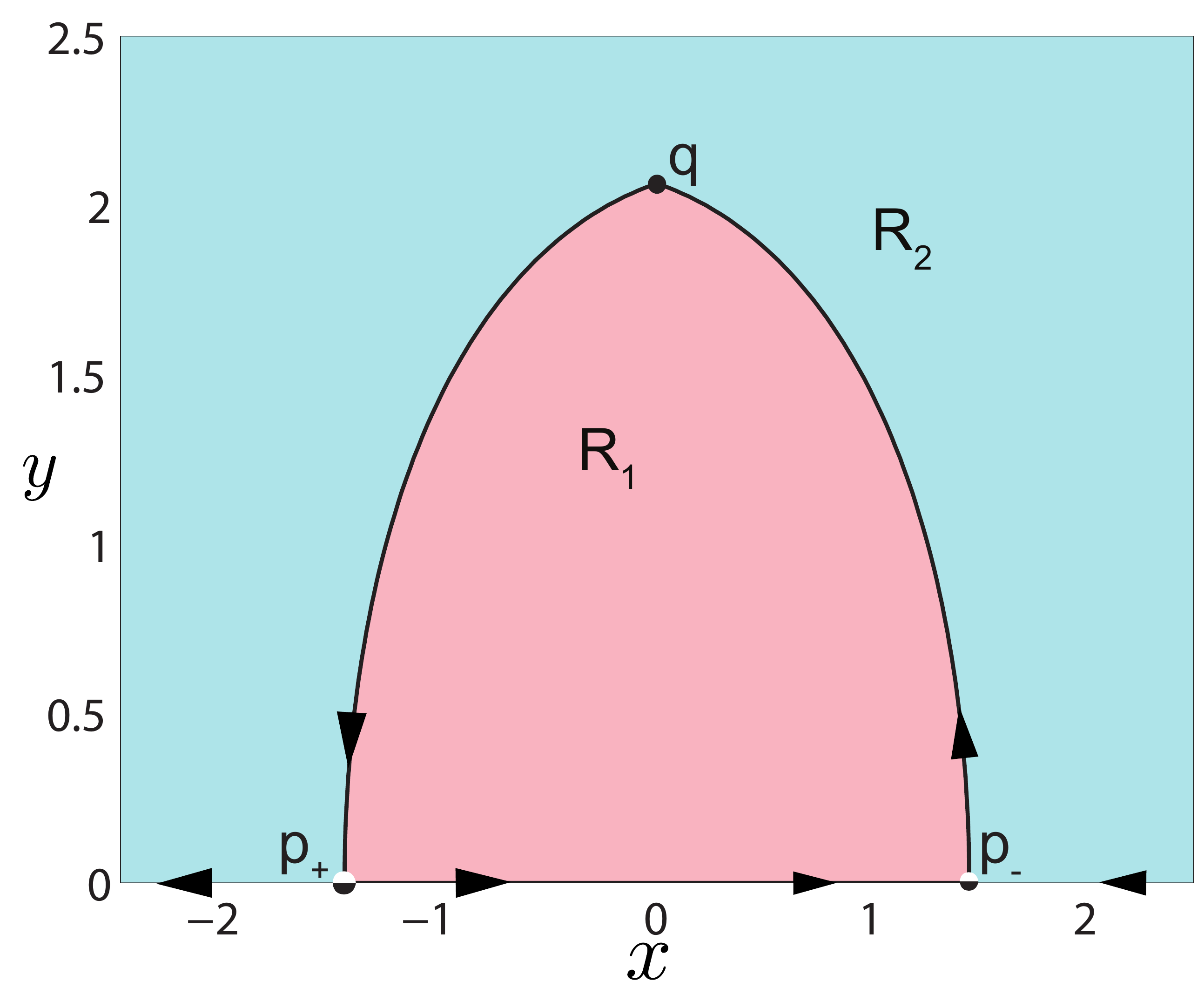}} 
	\caption{(a) Shows the closed streamlines and the free flow region for the unperturbed vortex pair flow. Even a simple sinusoidal perturbation changes the picture dramatically as transport between the regions $R_1$ and $R_2$ becomes feasible due to chaotic particle trajectories. (b) Shows the regions of interest $R_1$ and $R_2$ for studying transport in the OVP flow. The transport between the regions $R_1$ and $R_2$ can be explained and quantified in terms of the turnstile mechanism.}
	\label{fig:ovp_unperturbed}
\end{figure}







We can cast the equation of motion of a passive particle in the OVP flow in the form of Hamilton's equations using the stream function as the Hamiltonian function and performing perturbation expansion of the velocity field (see Ref.~\cite{Rom-kedar1990} for details). Thus, the system can be expressed as
\begin{equation}
\left.
\begin{aligned}
\d{x}{t} & = f_1(x,y)+ \epsilon g_1(x,y,t/\gamma;\gamma) + \bigO(\epsilon^2)\\
\d{y}{t} & = f_2(x,y) - \epsilon g_2(x,y,t/\gamma;\gamma) + \bigO(\epsilon^2)\\
\end{aligned}
\right.
\label{eqn:ovp_expansion_eps}
\end{equation}
where
\begin{equation}
\left.
\begin{aligned}
f_1 &= -\frac{y-1}{I_-} + \frac{y+1}{I_+} - 0.5 \\
f_2 &= \frac{x}{I_-} - \frac{x}{I_+}
\end{aligned}
\right.
\end{equation}
and
\begin{equation}
\left.
\begin{aligned}
g_1 &= [\cos(t/\gamma)-1]\left\{\frac{1}{I_-} + \frac{1}{I_+} - \frac{2(y-1)^2}{I_-^2} - \frac{2(y+1)^2}{I_+^2}\right\} \\
&+ (x/\gamma)\sin(t/\gamma)\left\{\gamma^2 \left[\frac{y-1}{I_-^2} -\frac{y+1}{I_+^2}\right] + 1 \right\} - 0.5 \\
g_2 &= 2x[\cos(t/\gamma)-1]\left\{\frac{y-1}{I_-^2} + \frac{y+1}{I_+^2} \right\} \\
&+ (1/\gamma)\sin(t/\gamma)\left\{ \frac{\gamma^2}{2}\left[\frac{1}{I_-} - \frac{1}{I_+}\right] -x^2\gamma^2\left[\frac{1}{I_-^2} - \frac{1}{I_+^2} \right] -y \right\} \\
\text{and} \qquad I_{\pm} &= x^2 + (y \pm 1)^2
\end{aligned}
\right.
\label{eqn:fg}  
\end{equation}
The OVP flow has two non-dimensional parameters which denote the circulation strength of the vortices, $\gamma$, and perturbation amplitude, $\epsilon$. 
For a complete analysis of transport in such a flow,\citeauthor{Rom-kedar1990}~\cite{Rom-kedar1990} studied a combination of these two parameters, but we will present two cases as application of the computational method implemented in \emph{Lober}. 
Due to the complicated nature of the velocity field in Eqn.~\eqref{eqn:ovp_expansion_eps}, we numerically compute the invariant manifolds using the \textit{globalization} technique (see Ref.~\cite{Koon2000a,Koon2011a,Naik2017}), and based on the benchmark algorithm in~\cite{Mancho2003}. The two cases considered here shows different manifold geometry, although topologically similar, as shown in Fig.~\ref{fig:ovp_case5af}. 
\begin{figure}[!ht]
	\centering
	\subfigure[$\gamma = 0.5$]{\includegraphics[width=0.45\linewidth]{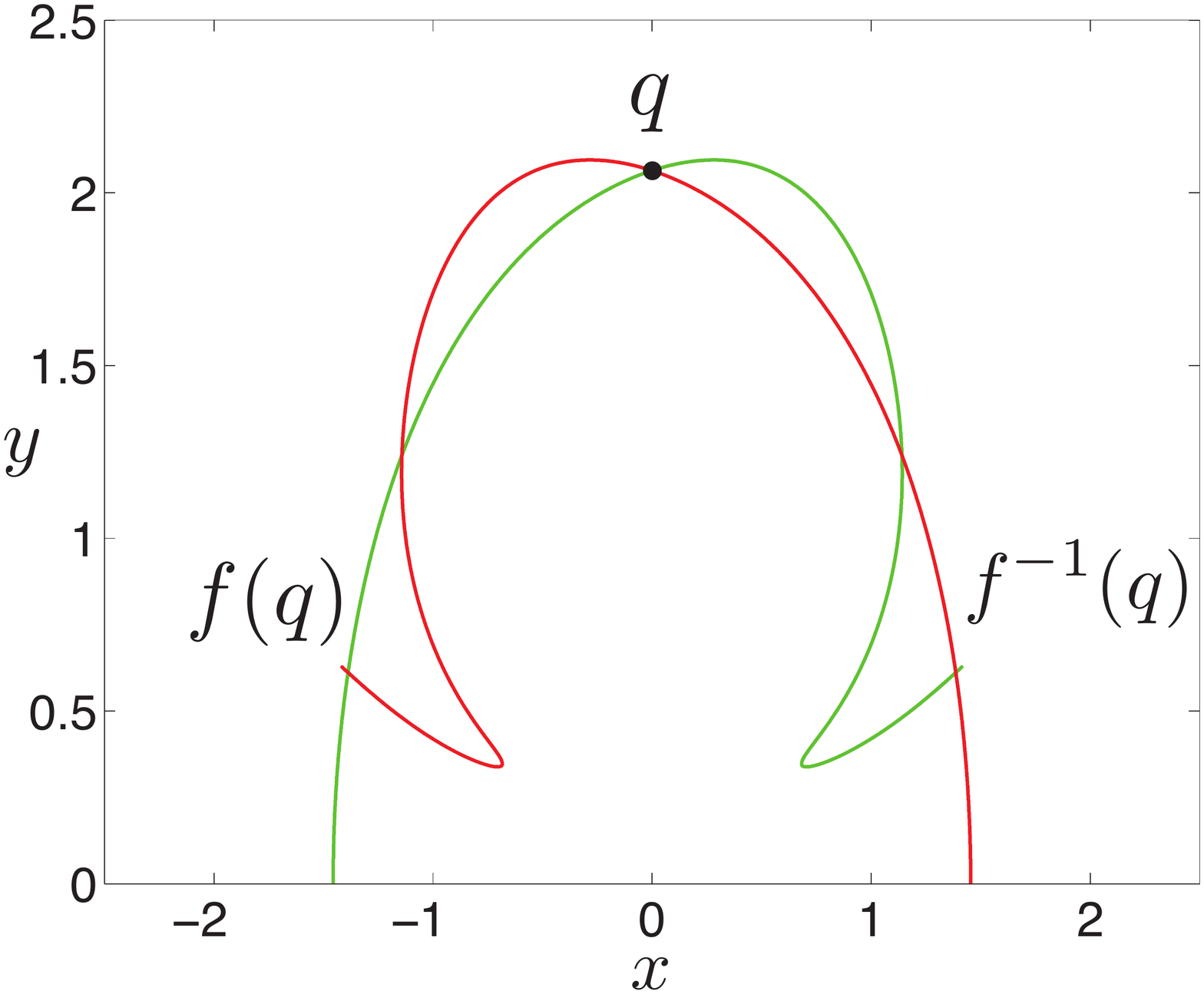}}
	\subfigure[$\gamma = 1.81$]{\includegraphics[width=0.45\linewidth]{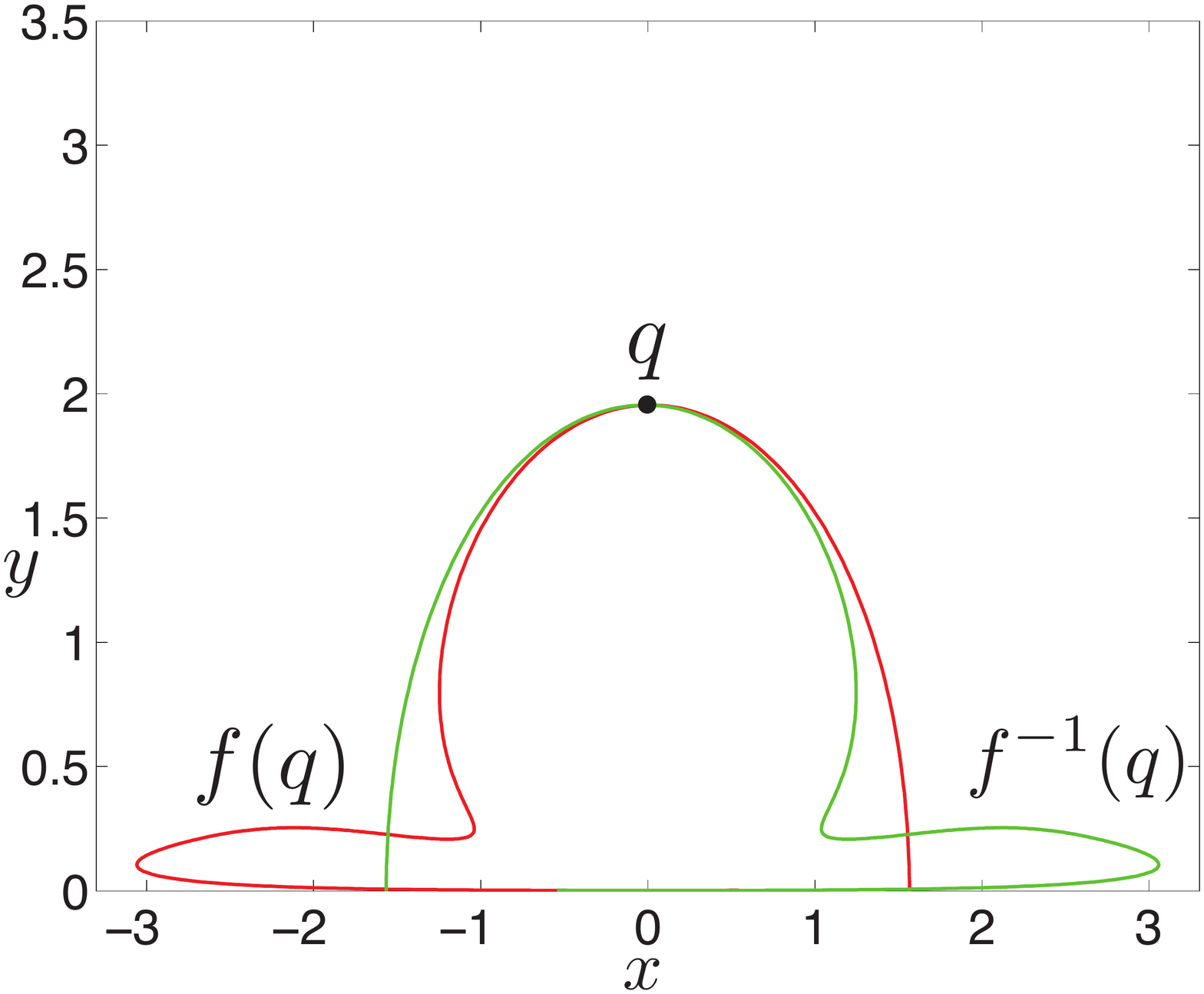}}
	\caption{Computed unstable(red) and stable(green) manifolds of the hyperbolic fixed points for perturbation amplitude of $\epsilon = 0.1$, (a) $\gamma=0.5$ and (b) $\gamma=1.81$. The manifolds are shown for qualitative comparison and a change in orientation for the $\gamma$ values as predicted by Melnikov theory in Ref.~\cite{Rom-kedar1990}.}
	\label{fig:ovp_case5af}
\end{figure}
We call the difference in geometry of \textit{turnstile}, the pair of lobes formed by the segments of manifolds between $q$ and $f^{-1}(q)$, by the near-orthogonal (shown in Fig.~\ref{fig:ovp_case5af}(a)) and near-tangent intersections (shown in Fig.~\ref{fig:ovp_case5af}(b)) that is generated in these two cases ($\gamma=0.5, 1.81$). Using the computed manifolds as inputs to \emph{Lober}, we identify the intersection points and intersection areas, that is pips and lobes, respectively, as shown in Fig.~\ref{fig:ovp_cases5}. 
\begin{figure}[!ht]
	\centering
	\subfigure[$\gamma = 0.5$]{\includegraphics[width=0.45\linewidth]{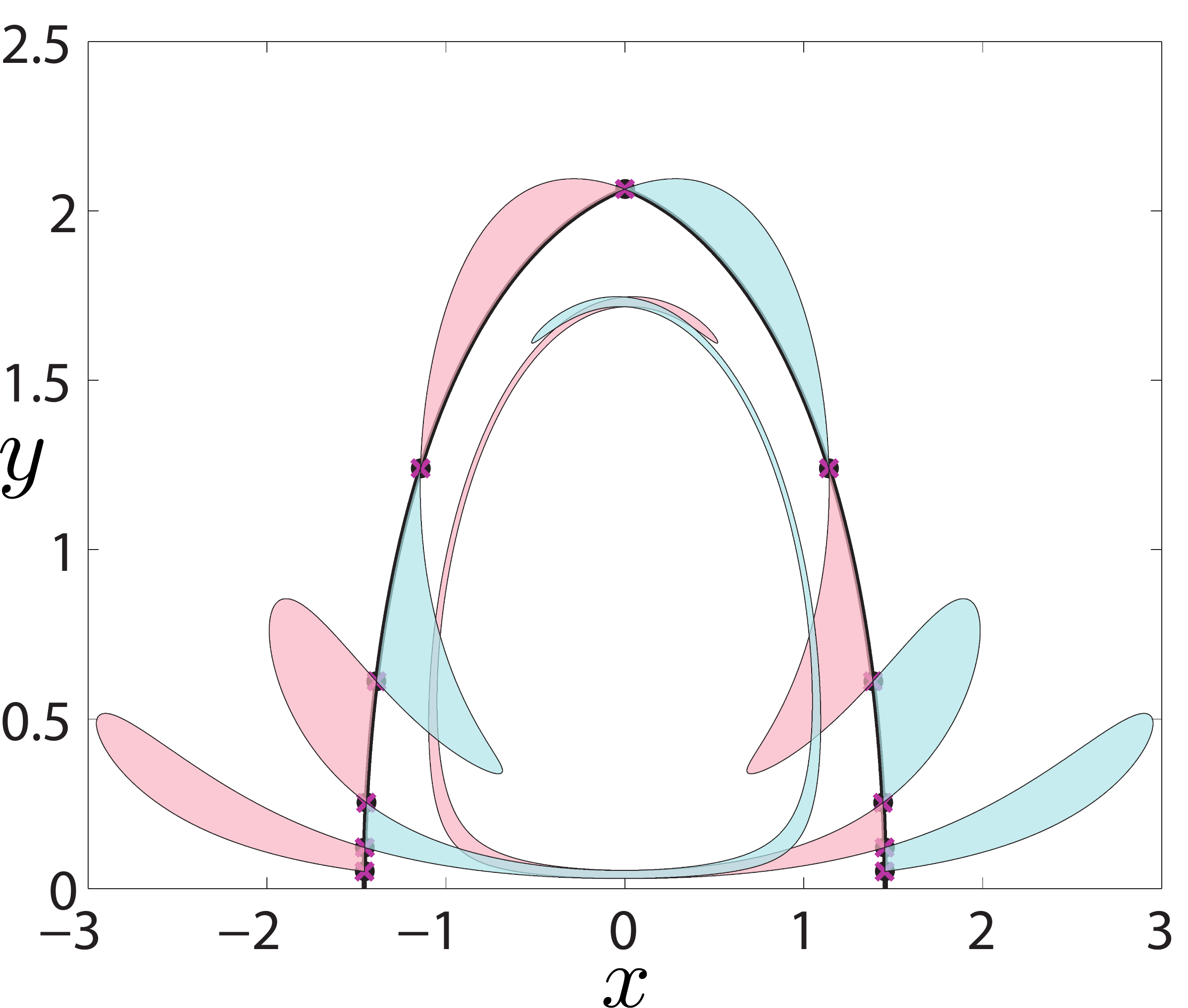}}
	\subfigure[$\gamma = 1.81$]{\includegraphics[width=0.45\linewidth]{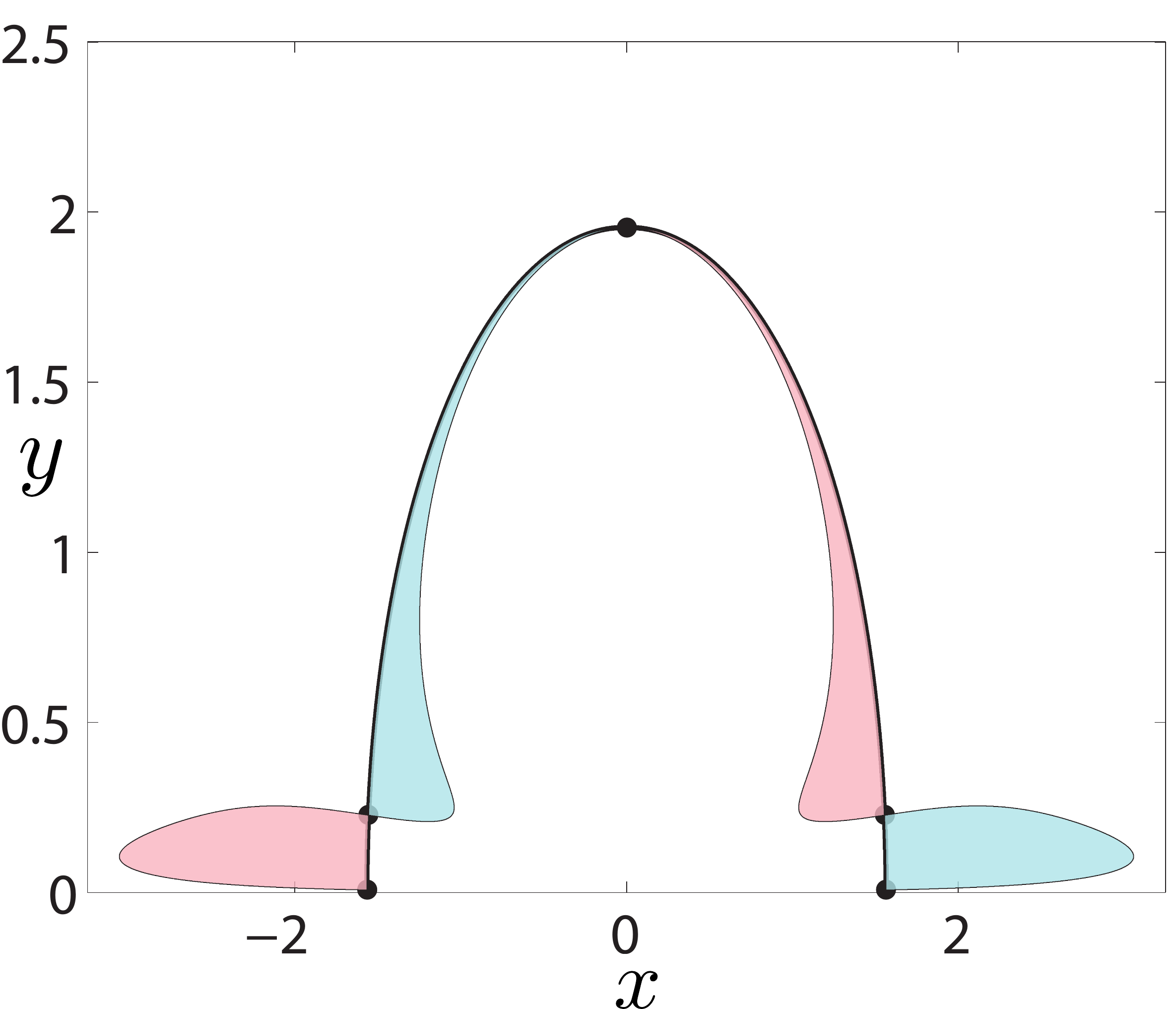}}
	\caption{Output of \emph{Lober} identifying the intersection points, that is pips shown as magenta cross on filled black circles, intersection areas, that is lobes shown as cyan and magenta filled area, and the boundary connecting the hyperbolic fixed points and the intersection point. The lobes colored as magenta map from region $R_1$ into $R_2$, and the ones that are colored cyan map from region $R_2$ to $R_1$ in successive iterates of the Poincar\'e map. For both the cases, $\epsilon=0.5$.}
	\label{fig:ovp_cases5}
\end{figure}
We note that this approach is based on computing the area of the turnstile and its iterates to compute the transport rate.

Another approach in phase space transport is to consider the iterates of the boundary parametrized using an intersection point, say $q$ as in Fig.~\ref{fig:ovp_unperturbed}(b), and the segments of stable/unstable manifolds connecting with the hyperbolic fixed points, say $p_+$ and $p_-$ as in Fig.~\ref{fig:ovp_unperturbed}(b). Thus, in case of time-dependent fluid flow, the boundary can be defined by 
\begin{equation}
B(q) = S[q,p_+] \cup U[q,p_-] 
\label{eqn:separ}
\end{equation}
where the point $q$ is the \emph{boundary intersection point} or \emph{bip} which is the first intersection of the manifolds computed for a given initial phase of the map. The bip parametrizes the boundary which acts as a partial barrier when perturbation is added to the flow. Furthermore, the turnstile can be defined in terms of the bip, and its pre-image as

\begin{equation}\label{eqn:turnstile}
L_{1,2}(1) \bigcup L_{2,1}(1) = S[f^{-1}(q),q] \bigcup U[f^{-1}(q),q]
\end{equation}
where $L_{i,j}(n)$ denotes a lobe in a region $R_i$ that is mapped to a region $R_j$ after n iterates, using the notation in Ref.~\cite{Wiggins1992chaotic} and we show a general schematic in Fig.~\ref{fig:regions}. We note here that this definition of turnstile lobe uses the pre-image of the bip, and requires computing the intersection of the boundary, Eqn.~\eqref{eqn:separ}, and its pre-images under the Poincar\'e map.
Thus, we can express the particles of species $S_i$ (material in the region $R_i$ as shown in Fig.~\ref{fig:regions}) that gets transported across the partial barriers in terms of intersection areas of the turnstile lobe, and iterates of the Poincar\'e map~\eqref{eqn:flow_map2D}. Following the theory described in \cite{Wiggins1992chaotic}, the quantities used in computation of transport rate are 
\begin{itemize}
	\item $a_{i,j}(n) :$ Flux of species $S_i$ from region $R_i$ into region $R_j$ on the $n^{th}$ iterate.
	\item $T_{i,j}(n) :$ Total amount of species $S_i$ contained in region $R_j$ after the $n^{th}$ iterate. 
\end{itemize}
We have the following approaches for computing these quantities.
\begin{enumerate}
	\item \textit{Boundary method:} Generate the boundary by selecting a point $q$ as bip, and compute the pre-images of this boundary along with the pre-images of $q$ ($f^{-1}(q),f^{-2}(q), \ldots$). This can be computed by evolving in backward time, or executing \emph{Lober} with an option to define the iterated boundary. This is shown in Fig.~\ref{fig:separ_BIPS}, and these curves are now used as input to the \emph{light} option of \emph{Lober} (see command options in~\ref{sect:usage}).
	
	\item \textit{Lobe method:} Obtain the turnstile lobes $L_{1,2}(1),L_{2,1}(1)$, and images of entraining (in the sense of entering the region $R_1$ in the next iterate) turnstile lobe $L_{2,1}(1)$ or pre-images of detraining turnstile lobe $L_{1,2}(1)$. These are obtained from \emph{Lober} as closed curves and given as input to the \emph{light} option of \emph{Lober} to compute intersection areas since the lobes and its iterate can form non-transverse intersections and may require the \emph{densifier} option.
\end{enumerate}
\begin{figure}[!ht]
	\centering
	\includegraphics[width=0.95\textwidth]{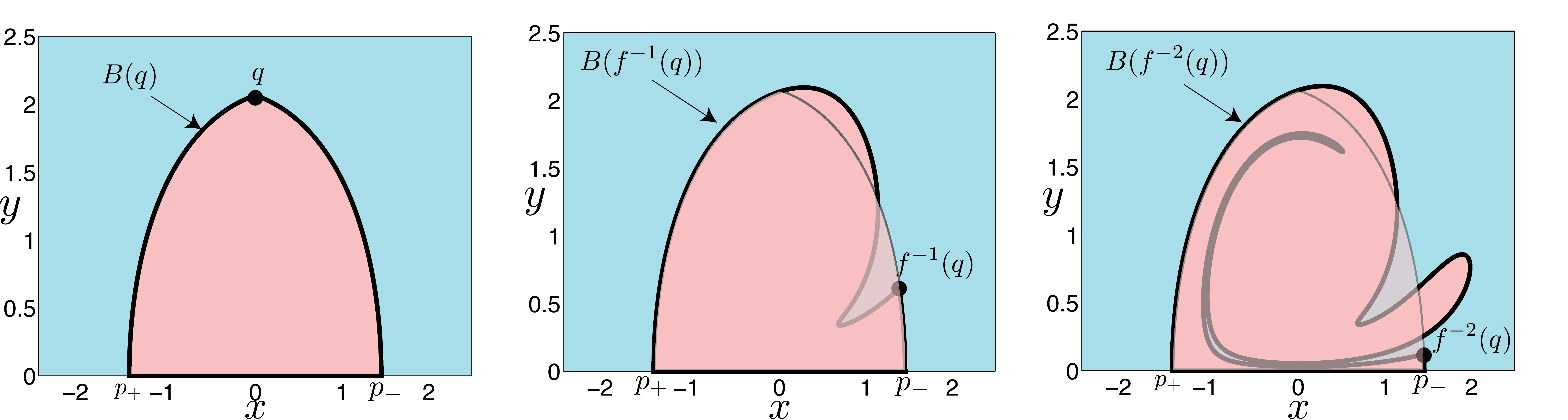} 
	\caption{Showing the boundary with $q$ as a bip and its backward evolution with the pre-images of $q$ as bips. The region $R_1$ is the transparent layer on the pre-images.}
	\label{fig:separ_BIPS}
\end{figure}
Thus, using the Lemma 2.3 and Theorem 2.5 in Ref.~\cite{Wiggins1992chaotic}, the quantity $a_{\rm 2,1}(n)$ in terms of the turnstiles and intersection areas can be expressed as
\begin{align}
a_{\rm 2,1}(n) = T_{2,1}(n) - T_{2,1}(n-1) = \mu(L_{2,1}(1)) - \sum_{m=1}^{n}\mu( L_{2,1}(1)  \cap f^{m-1}(L_{1,2}(1)) )
\label{eqn:aij}
\end{align}
where the quantity under the sum denotes the intersection area of the lobe $L_{2,1}(1)$ and pre-images of $L_{1,2}(1)$, and is shown in the Fig.~\ref{fig:lobeD_preimages_intersect_lobeE}. 
\begin{figure}[!ht]
	\centering
	\includegraphics[width=1.0\textwidth]{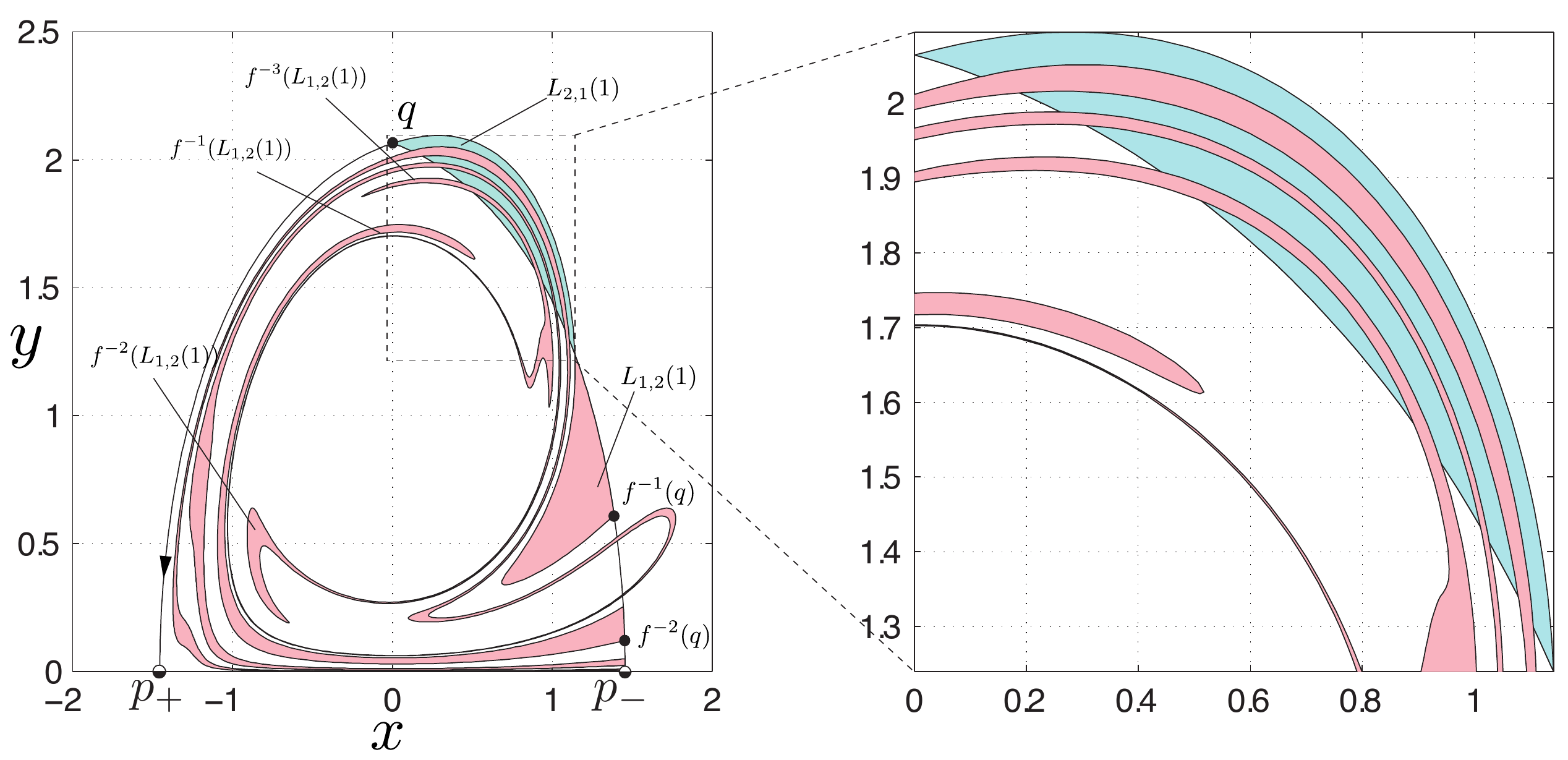}
	\caption{Shows the \textit{turnstile} lobe in cyan (entraining lobe labeled $L_{1,2}(1)$), magenta (detraining lobe labeled $L_{2,1}(1)$) and its 3 pre-images for the OVP flow. The zoom-in view shows the intersection regions that are used in quantifying the transport. We have used $\epsilon=0.1,\gamma=0.5$.}
	\label{fig:lobeD_preimages_intersect_lobeE}
\end{figure}
While, the lobe method is a reduced order calculation for transport, the boundary method is useful when multilobe, self-intersecting turnstile generate near-tangent intersections. In terms of the boundary, the quantity $T_{1,2}(n)$ is given by
\begin{align}
T_{1,2}(n) = [B(q) \setminus (B(q) \cap f^{-n}(B(q)))]
\end{align}
and the quantity $a_{1,2}(n)$ is
\begin{align}
a_{1,2}(n) = T_{1,2}(n) - T_{1,2}(n-1)
\end{align}
For validating the output from \emph{Lober}, invariant manifolds of the hyperbolic fixed points $p_{+}, p_{-}$ are provided as input and the outputs are shown in Fig.~\ref{fig:ovp_cases5}. It clearly identifies the lobes, pips, and the boundary parametrized by the pip at $q = (0.0,2.065)$, and indexed as $\lfloor\# \; \text{pips}/2\rfloor$ for a given number of pips.
Using the boundary method, we obtain the turnstile lobe area, $\mu(L)$, for different circulation strength of the vortices, $\gamma$, values, and show the numerical result in Fig.~\ref{fig:lobe_area_gamma}. However, for validation purpose we have performed the intersection of lobe area computations to obtain $a_{ij}$ in Eqn.~(\ref{eqn:aij}) and show the results for first 7 iterations of the map $f$ in Fig.~(\ref{fig:ak_vs_k}).
\begin{figure}[!ht]
	\centering
	\subfigure[]{\includegraphics[width=0.45\textwidth]{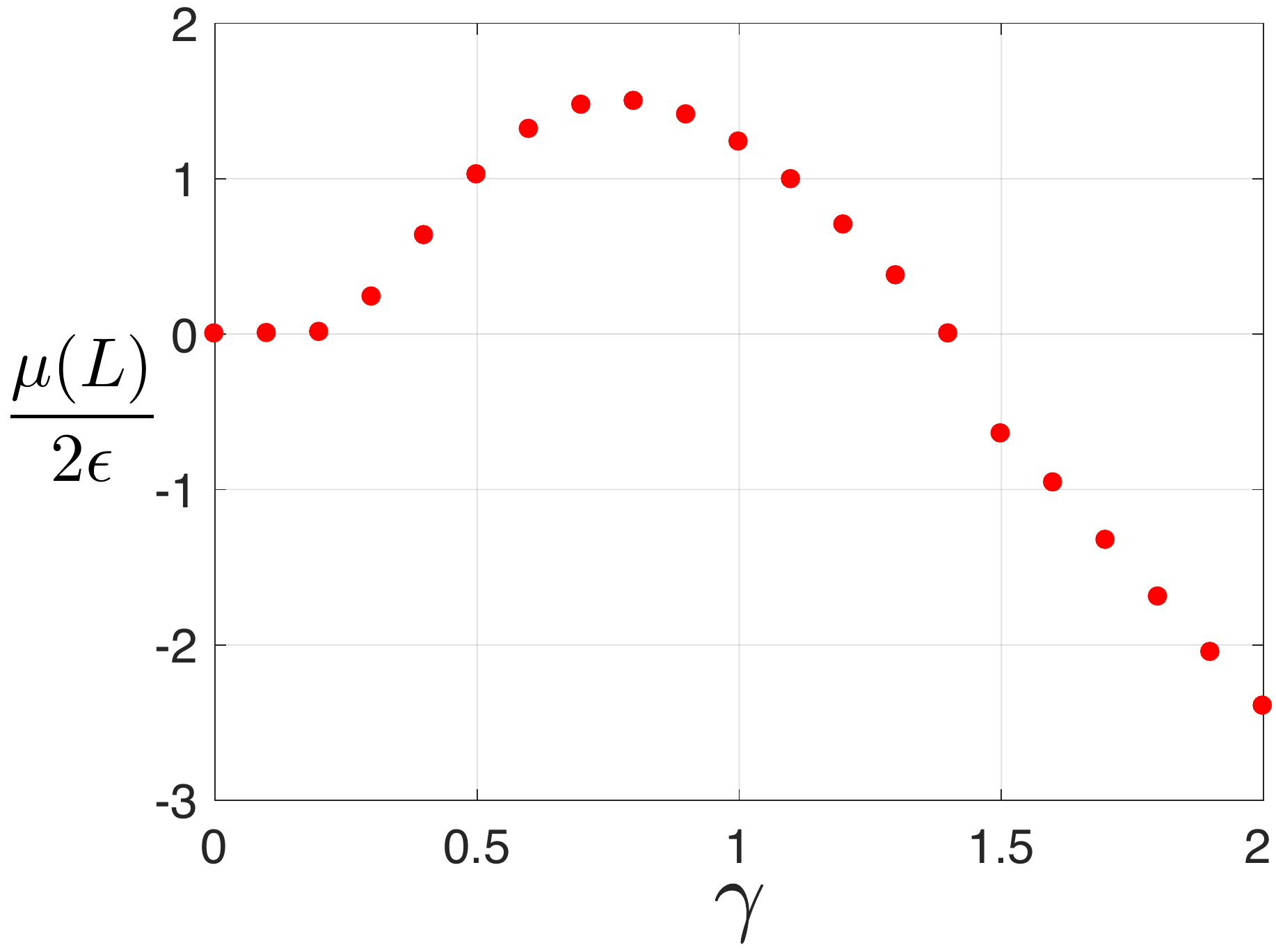}\label{fig:lobe_area_gamma}}
	\subfigure[]{\includegraphics[width=0.45\textwidth]{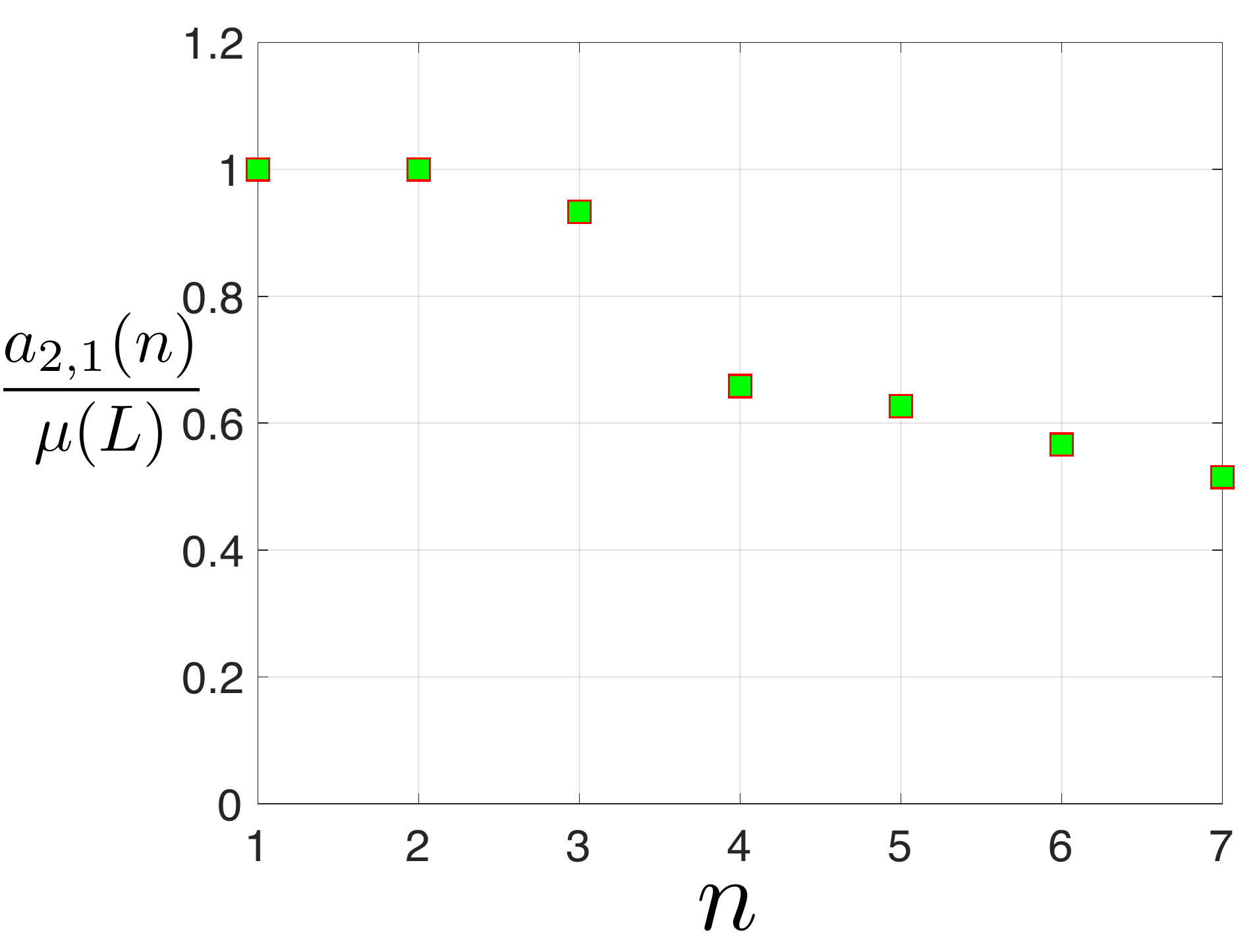}\label{fig:ak_vs_k}}
	\caption{(a) Shows the area of lobes for different $\gamma$ values with $\epsilon = 0.1$. It is to be noted that for $\gamma \geqslant 1.5$, the lobe area is shown as negative to imply the change in the geometry of the manifold intersection. This agrees well with Fig.~9 in Ref.~\cite{Rom-kedar1990} which compares the brute force lobe area calculations with Melnikov function. (b) Shows the normalized fluid volume that given by the quantity $a_{\rm 2,1}(n)$ in Eqn.~\eqref{eqn:aij} for 7 iterates of the map $f$, and for perturbation period $\gamma = 0.5$ and $\epsilon = 0.1$.}
\end{figure}
 
\subsection{Escape from a potential well} \label{sect:escape}

In this section, we will apply the numerical method to computation of rate of escape from a potential well in the context of capsize of a ship studied as a nonlinear coupling of roll and pitch degrees of freedom. The full description and analysis of this system can be found in Ref.~\cite{Naik2017} which applies tube dynamics~\cite{Koon2000a}, a geometric method of phase space transport, to study capsize of a ship in a 2 degree of freedom (DOF) system. The ship capsize model can be considered as an archetype for studying escaping dynamics in an engineering problem, and hence the rate of escape from the underlying potential well provides estimate of probability of capsize. As derived in Ref.~\cite{Naik2017}, the rescaled Lagrangian is given by
\begin{align}
\mathcal{L}(x,y,v_x,v_y) &= \frac{1}{2}v_x^2 + \frac{1}{R^2}v_y^2 - V(x,y) \label{eqn:rescale_Lag}\\
\text{where,} \qquad \qquad V(x,y) &= \frac{1}{2}x^2 + y^2 - x^2y \label{eqn:rescale_pe}
\end{align}
is the underlying potential energy, and is shown in Fig.~\ref{fig:eff_pot_colormap_5energy_levels}. 
\begin{figure}[!ht]
	\centering
	\includegraphics[width=0.45\textwidth]{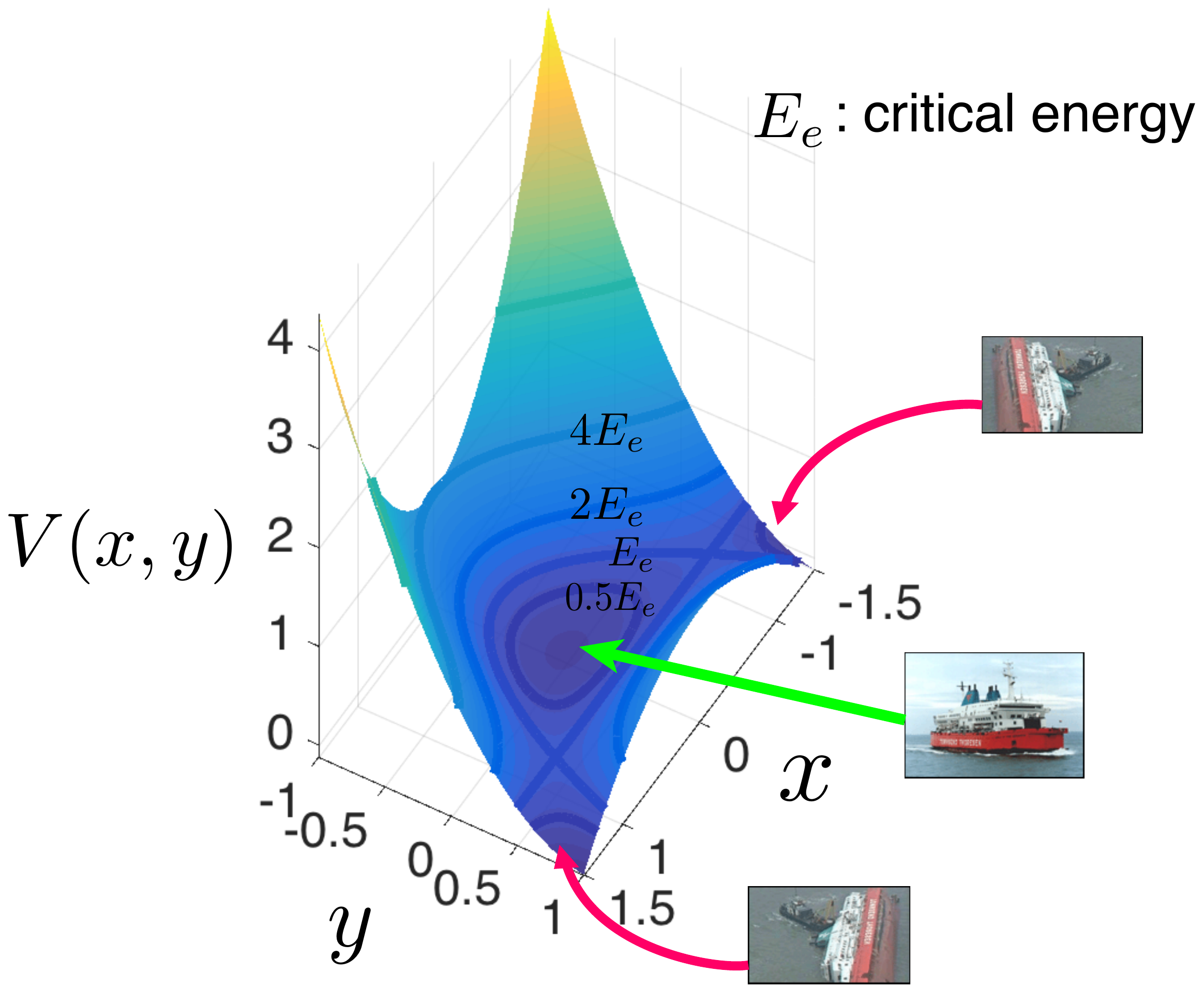}
	\caption{Shows the effective potential energy with an upright ship in the region that corresponds to bounded motion inside the well and a capsized ship in the region that corresponds to unbounded motion. The total energy of the system can be considered as fixing a height of this potential well and shown here as contour lines of $0.5E_e, E_e, 2E_e, 4E_e$ on the configuration space $(x,y)$ for different values above and below the critical energy $E_e$.}
	\label{fig:eff_pot_colormap_5energy_levels}
\end{figure}
Using the Lagrangian Eqn.~\ref{eqn:rescale_Lag}, the equation of motion is given by
\begin{equation}
\left.
\begin{aligned}
\ddot{x} &= -x + 2xy \\
\ddot{y} &= -R^2y + \frac{1}{2}R^2x^2
\label{eqn:rescale_EOM}
\end{aligned}
\right.
\end{equation}
In this rescaled form, the only parameter is $R=\omega_{\theta}/\omega_{\phi}$, which is the ratio of natural pitch and roll frequencies. 
The rescaling of the coordinates and time has made the nonlinear coupling term of the potential into unity and the original coordinates can always be recovered using the transformation
\begin{equation}
x = \frac{\phi}{\phi_e}, \qquad y = \frac{\theta}{2\theta_e}, \qquad \bar{t} = \omega_{\phi}t \label{eqn:rescaling_coords}
\end{equation}
We note that Eqn.~\ref{eqn:rescale_EOM} are identical to those derived by Ref.~\cite{Thompson1996}, where they were called \emph{symmetric internal resonance} equations. The potential energy~\eqref{eqn:rescale_pe} is also similar to the Barbaris potential studied by the chemistry community in Refs.~\cite{Sepulveda1994a,Babyuk2003,Barrio2009a}. 

The equation of motion of a ship in absence of non-conservative, time-dependent wave forcing is thus given by the first-order ordinary differential equations as
\begin{equation}
\left.\begin{aligned}
\dot{x} &= v_x \\
\dot{y} &= v_y \\
\dot{v}_x &= -x + 2xy \\
\dot{v}_y &= -R^2y + \frac{1}{2}R^2x^2  
\end{aligned}\right.
\label{eqn:capsizeODE_1storder}
\end{equation}
which conserves the energy 
\begin{align}
E(x,y,v_x,v_y) = \frac{1}{2}v_x^2 + \frac{1}{R^2}v_y^2 + \frac{1}{2}x^2 + y^2 - x^2y \label{eqn:cons_sys_energy}
\end{align}
Let $\mathcal{M}(e)$ be the energy surface given by setting the total energy~\eqref{eqn:cons_sys_energy} equal to a constant, that is,
\begin{align}
\mathcal{M}(e) = \{(x,y,v_x,v_y) | E(x,y,v_x,v_y) = e \} \label{eqn:energy_surf}
\end{align}
where $e$ denotes the constant value of energy. The projection of energy surface onto the configuration space $(x,y)$ is historically known as \textit{Hill's region}, and defines the region that is energetically accessible, and is shown in gray in Fig.~\ref{fig:hills_region}.
\begin{figure}[!ht]
	\centering
	\subfigure[]{\includegraphics[width=0.3\textwidth]{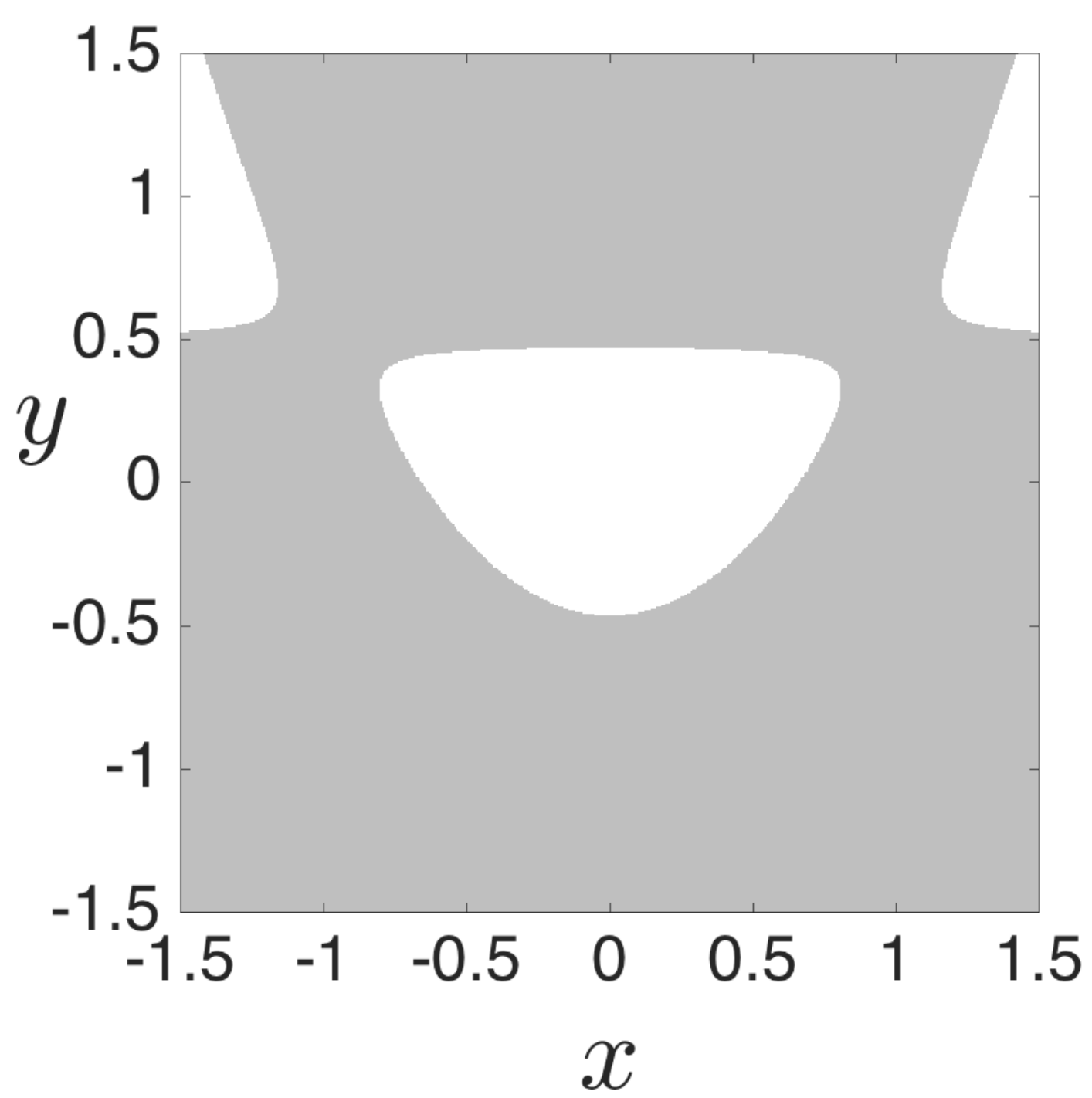}\label{fig:hills_region_energy22e-2}}
	\subfigure[]{\includegraphics[width=0.3\textwidth]{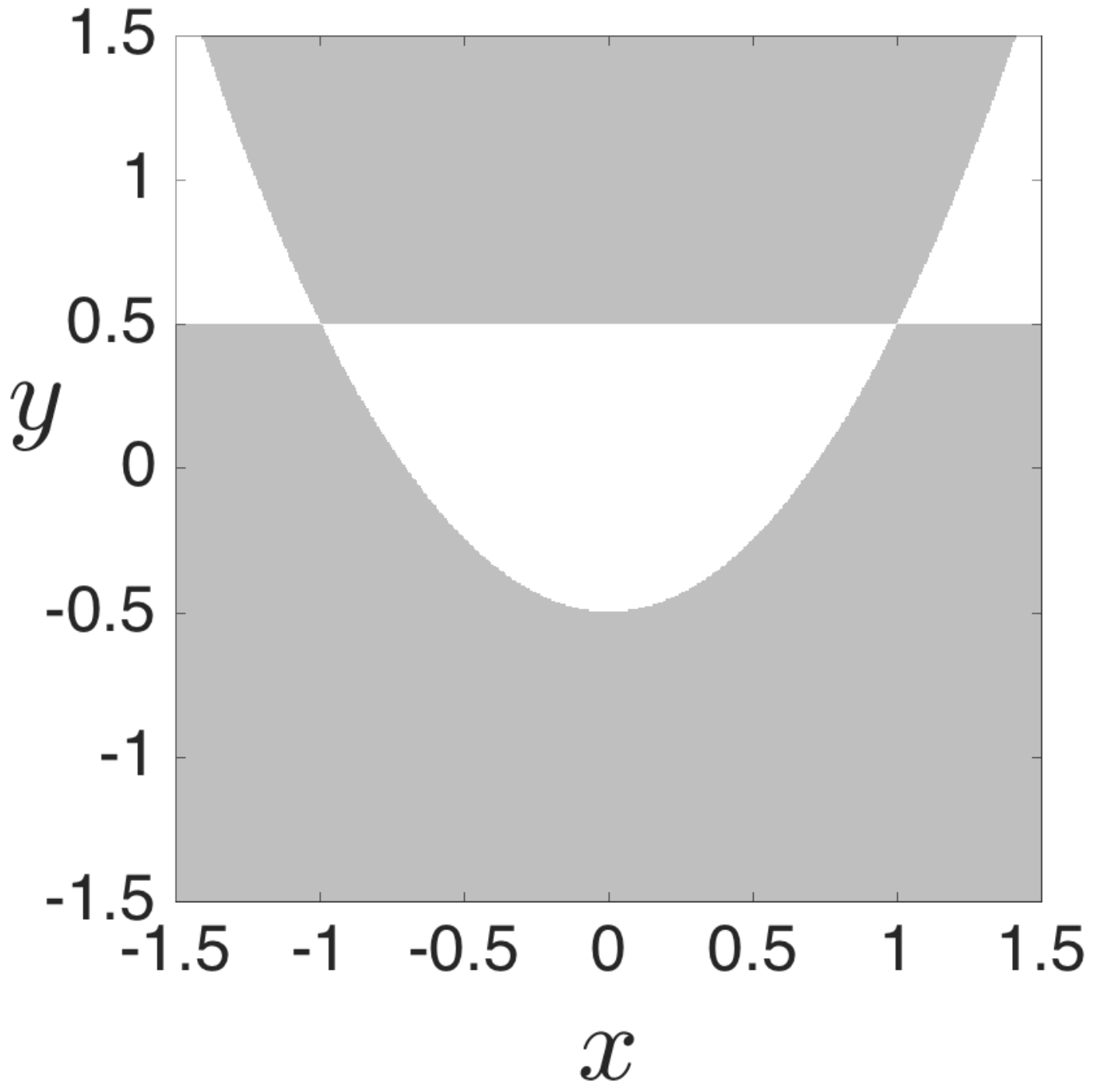}\label{fig:hills_region_energy25e-2}}
	\subfigure[]{\includegraphics[width=0.3\textwidth]{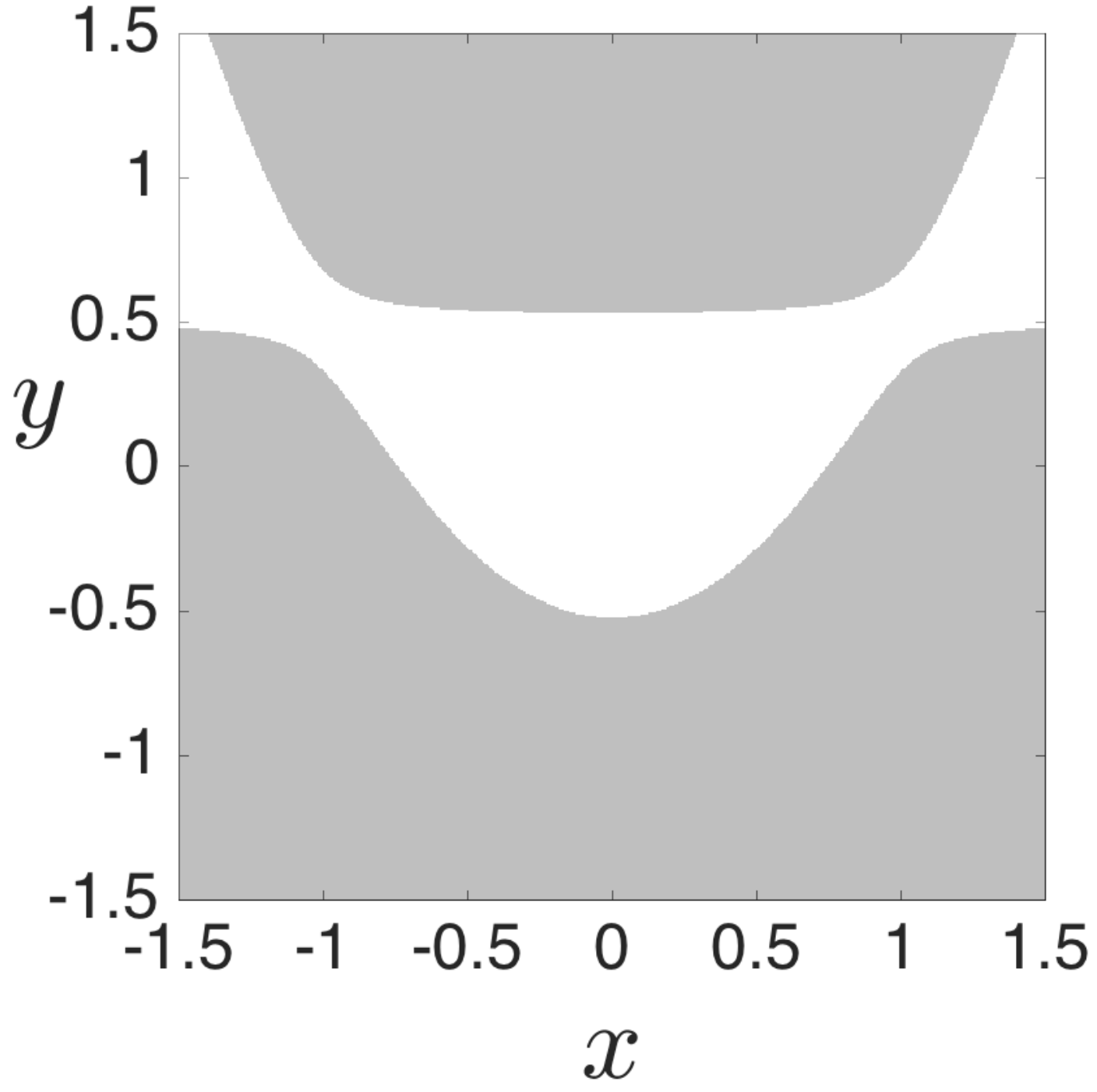}\label{fig:hills_region_energy28e-2}} 
	\caption{Show the Hill's region for (a) $e < E_e$, (b) $e = E_e$, and (c) $e > E_e$ where $E_e$ denotes the critical energy. The white region is the energetically accessible region bounded by the zero velocity curve, while the gray region is the energetically forbidden realm where kinetic energy is negative and motion is impossible.}
	\label{fig:hills_region}
\end{figure}
For a fixed energy $e$, the surface $\mathcal{M}(e)$ is a three-dimensional surface embedded in the four-dimensional phase space, $\mathbb{R}^4$. Furthermore, we can consider a cross-section of this three-dimensional surface, using a two-dimensional Poincar\'e surface-of-section (SOS) given by
\begin{equation}
U_1 = \left\{(y,v_y) | x = 0, v_x(y,v_y;e) > 0 \right\} , \qquad \text{motion to the right}  \label{eqn:sos_U1}
\end{equation}
The SOS can be used to define a two-dimensional ($\mathbb{R}^2 \rightarrow \mathbb{R}^2$) return map 
\begin{equation}
g : U^1 \rightarrow U^1 \label{eqn:return_map2D}
\end{equation}
for a constant energy, $e$, and $v_x > 0$ is used to enforce a directional crossing of the surface. This is shown for successively increasing energy in Fig.~\ref{fig:sosU1_diff_energy} for the SOS~\eqref{eqn:sos_U1}.
\begin{figure}[!h]
	\centering
	\subfigure[$e=0.22$]{\includegraphics[width=0.32\textwidth]{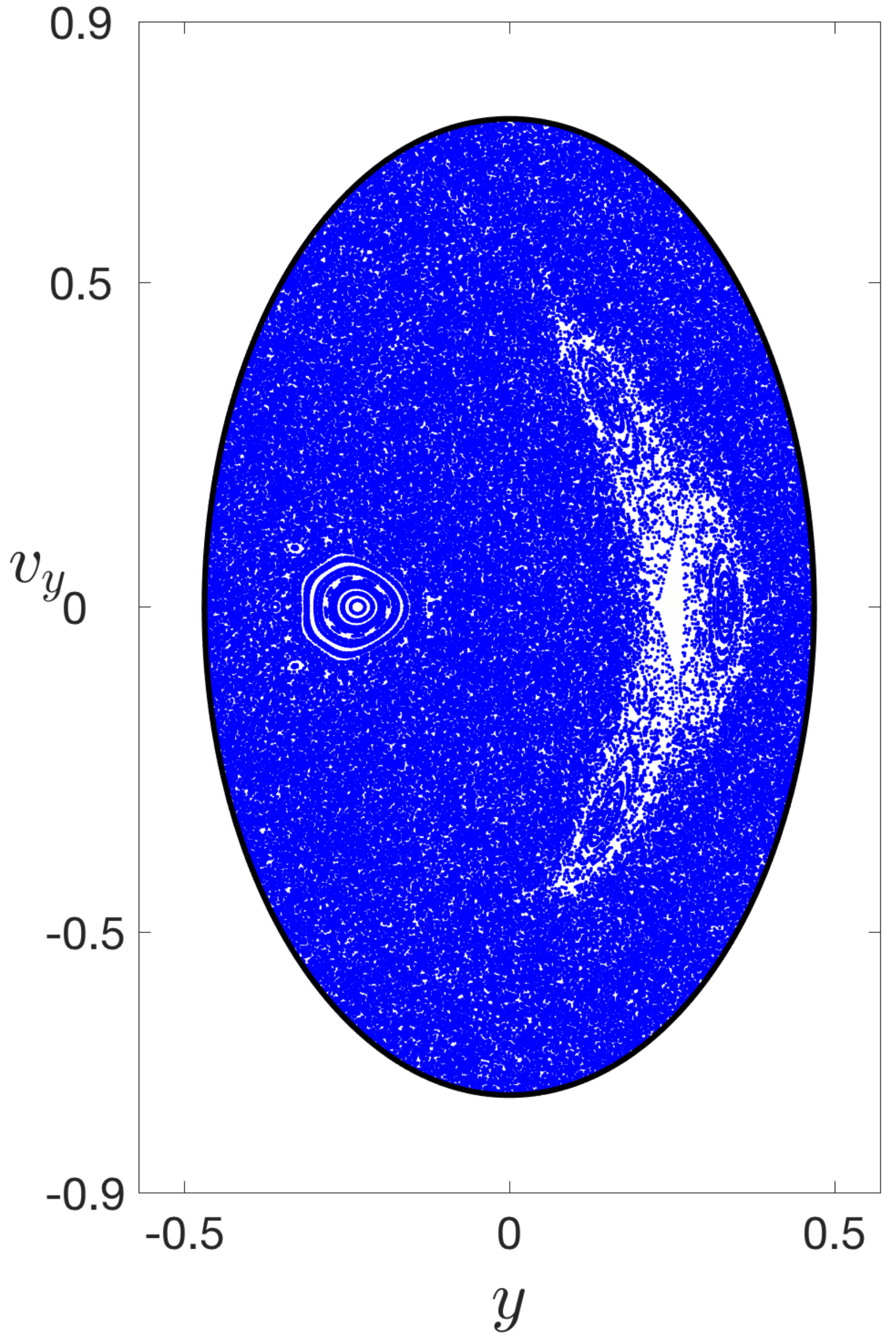}\label{fig:sos_energy22e-2}}
	\subfigure[$e=0.25$]{\includegraphics[width=0.32\textwidth]{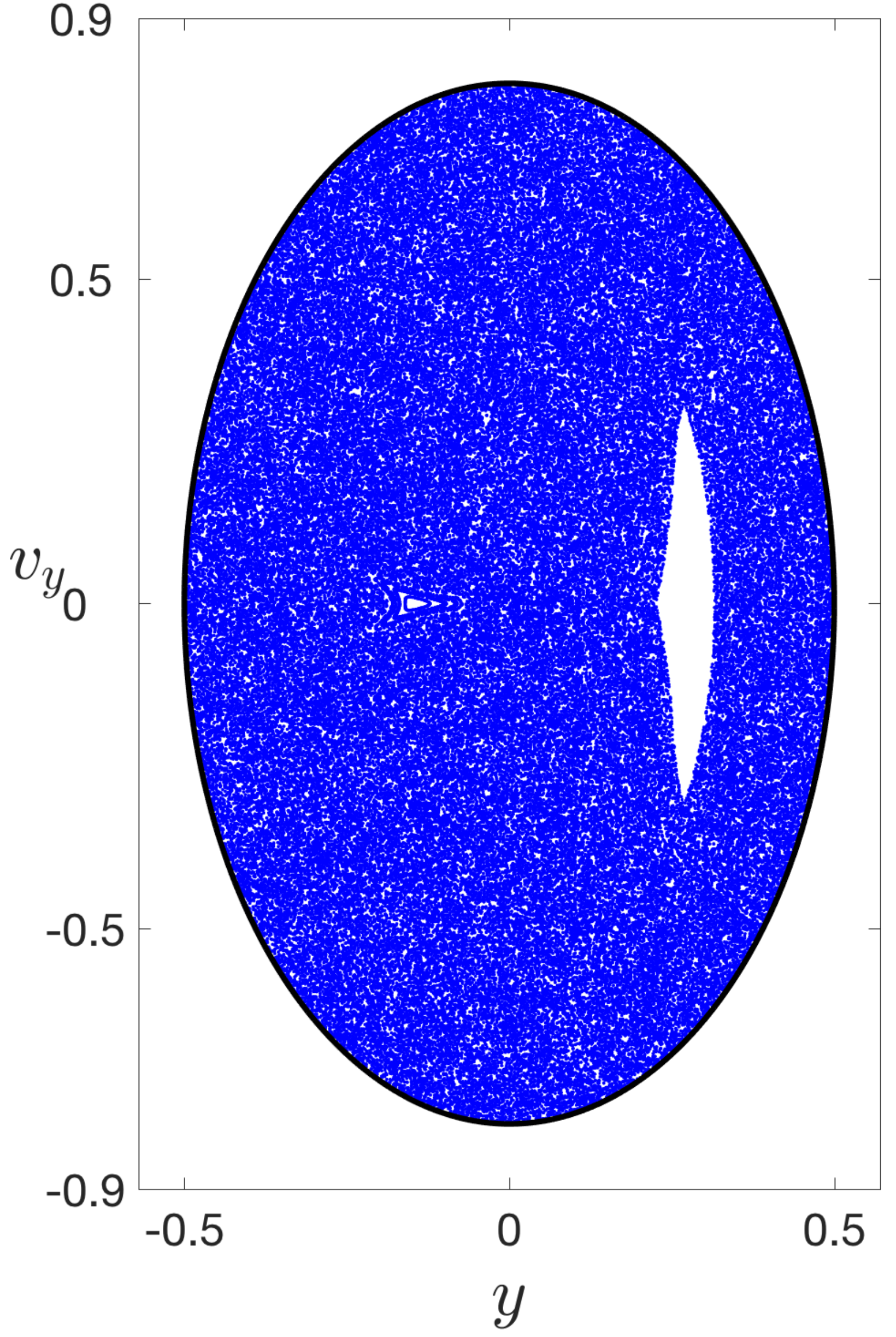}\label{fig:sos_energy25e-2}}
	\subfigure[$e=0.28$]{\includegraphics[width=0.31\textwidth]{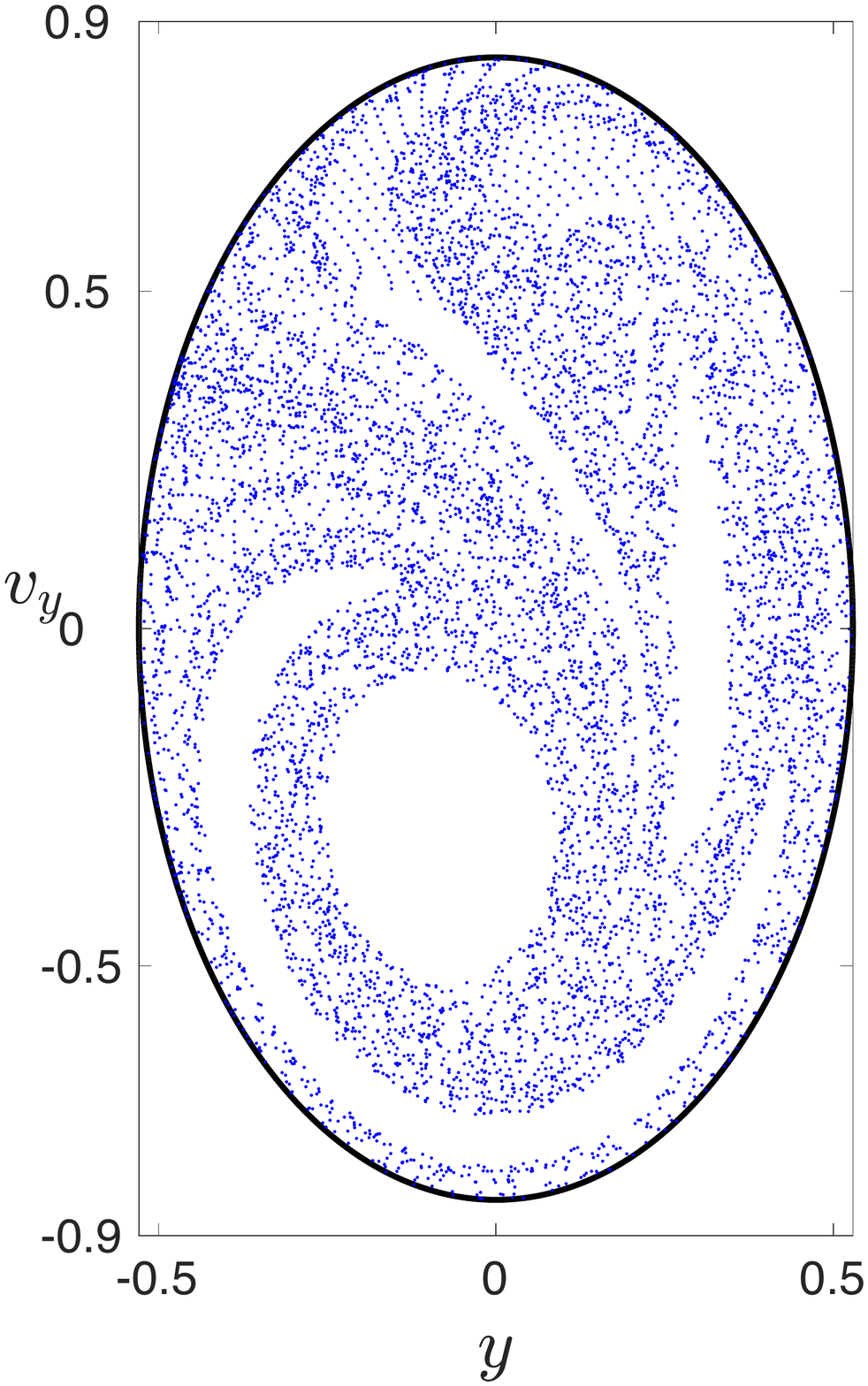}\label{fig:sos_energy28e-2}}
	\caption{Poincar\'e surface-of-section (SOS) of the energy surface showing orbits of the return map for different energy values in (a) $e=0.22$, (b) $e=0.25$, and (c) $e=0.28$. In the absence of damping and wave forcing, the system conserves energy. When energy is below the critical value, $E_e$, trajectories can not escape the potential well, and must intersect the SOS given by Eqn.~\eqref{eqn:sos_U1}. When the energy is above the critical value, trajectories will escape via the right or left saddles, and do not intersect the surface~\eqref{eqn:sos_U1} infinitely often. $R = 1.6$ is used in all the cases.}
	\label{fig:sosU1_diff_energy}
\end{figure}

The energy of the equilibrium points is called the critical energy (or escaping energy) which is given by $E_e = 0.25$. As the trajectories approach this energy from below, capsize becomes inevitable and this can be interpreted in terms of the potential energy well. Since the potential energy~\eqref{eqn:rescale_pe} is independent of any system parameter, the discussion based on this potential well will be more general. 

The dynamical system given by Eqn.~\eqref{eqn:capsizeODE_1storder} has saddle equilibrium points at $(\pm 1, 0.5, 0, 0)$ and the realms of possible motion in configuration space between the saddles is called the \textit{non-capsize realm}, and all possible states beyond as \textit{capsize realm}. When the energy is above the critical energy, $E_e$, bottlenecks appear around the saddle points (as shown in Fig.~\ref{fig:hills_region_energy28e-2}) that acts as partial barriers of capsize because trajectories escaping the potential well, in absence of forcing, reside inside the cylindrical manifolds of geometry $\mathbb{R}^1 \times \mathbb{S}^1$. Traditionally, escape rate of trajectories from the potential well, that is the non-capsize realm, is computed using \textit{tube dynamics}~\cite{Koon2000a,Koon2011a,Dellnitz2005}, or \textit{transition state theory}~\cite{jaffe2002statistical} along with Monte-Carlo method to estimate area that escape in a given time interval. 

However, we will use the implementation in \emph{Lober} to compute the escape rate from the potential well, and briefly summarize the steps involved in computing the curves that enclose the escaping regions on the SOS (see Ref.~\cite{Naik2017} for details). 

\textbf{a.} We select a suitable Poincar\'e SOS which is intersected by almost all of the trajectories that escape from the potential well. It is discussed in Ref.~\cite{Jaffe1999} as selecting a periodic orbit dividing surface that avoids pathological intersection of manifolds and Poincar\'e SOS. However, we make an educated guess for the present system, and use the SOS given by Eqn.~\eqref{eqn:sos_U1}, which defines a plane to capture motion of trajectories to the right, and shown as magenta plane in Fig.~\ref{fig:3D_tubes_mod7_v1}(a). 
\begin{figure}[!ht]
	\centering
	\subfigure[]{\includegraphics[width=0.4\textwidth]{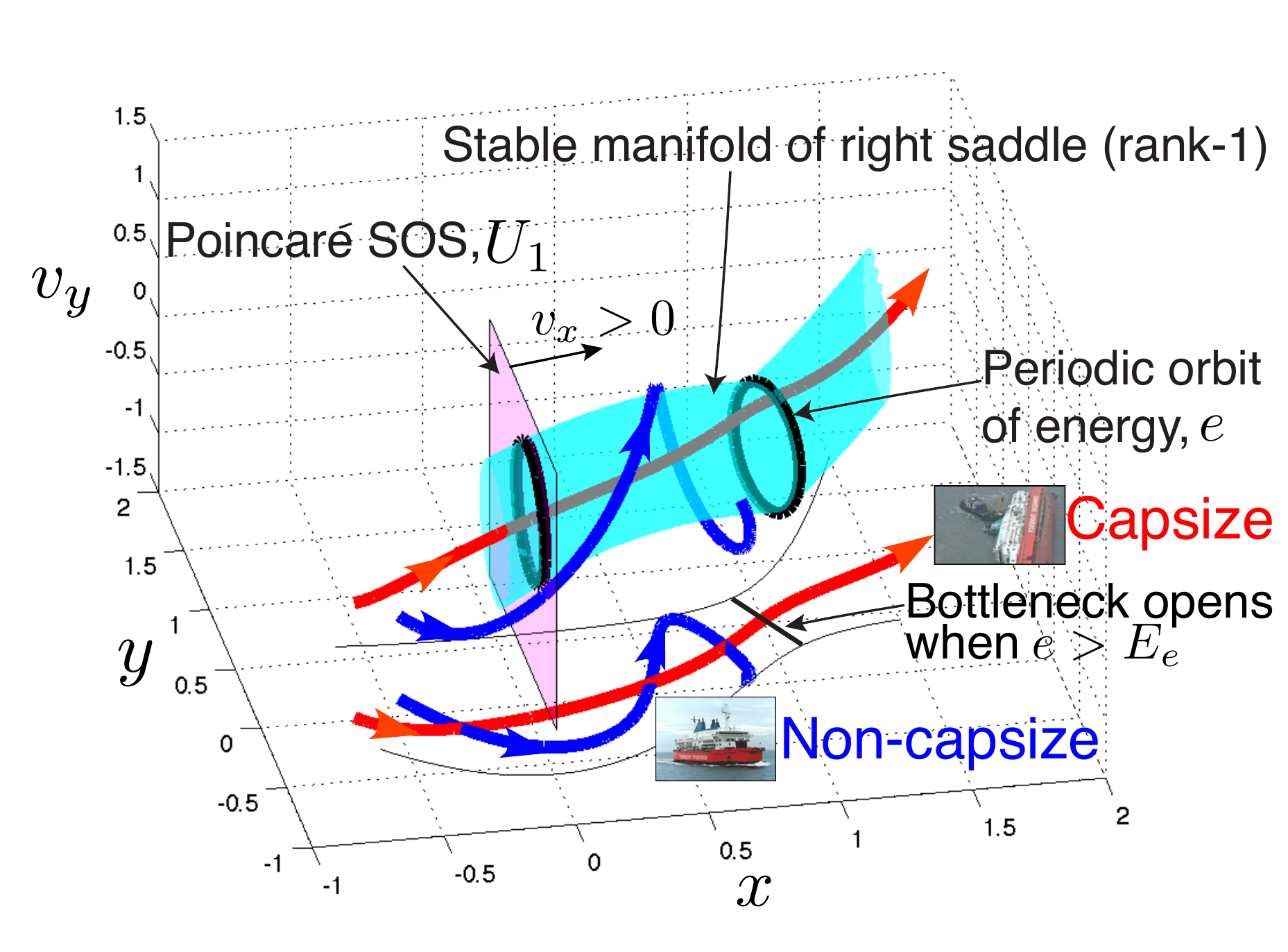}}
	\subfigure[]{\includegraphics[width=0.55\textwidth]{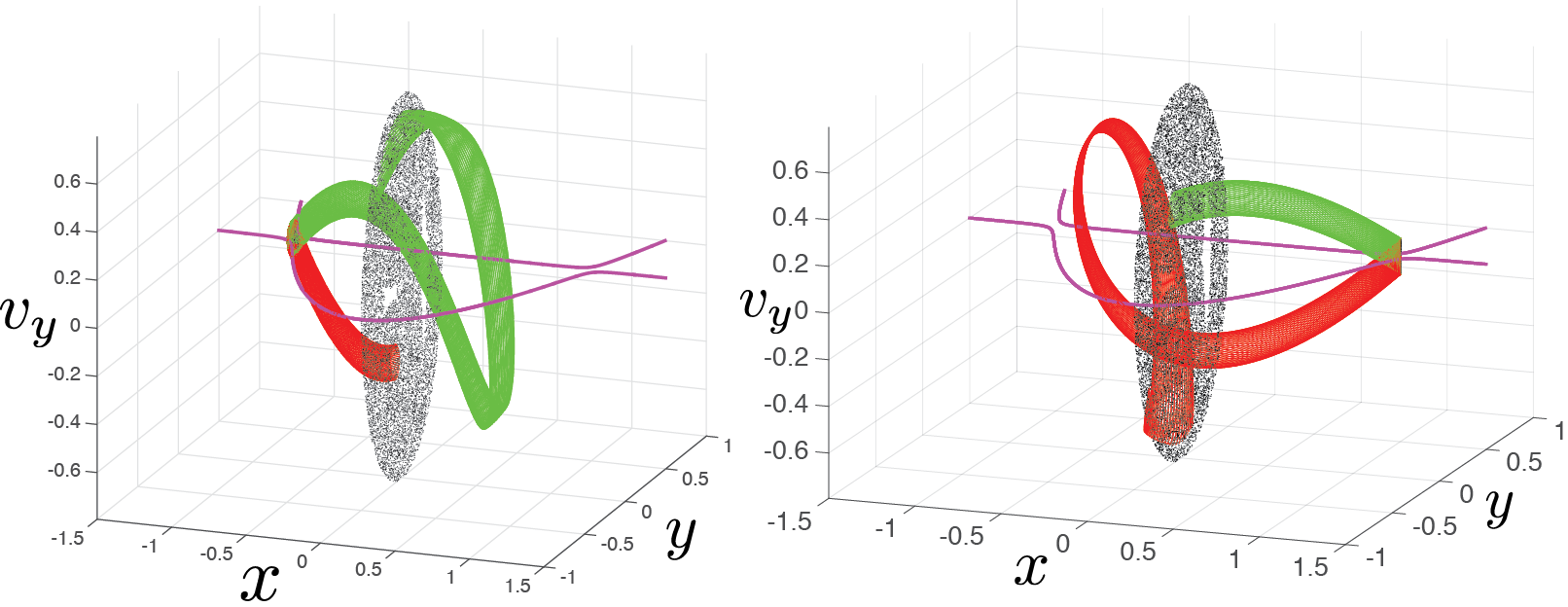}}
	\caption{(a) Shows the Poincar\'e SOS, $U_1$~\eqref{eqn:sos_U1}, as the magenta plane and the stable manifold of right saddle is shown as the cyan surface for the energy $e = 0.28$. The stable manifold of a given saddle up to its first intersection with $U_1$ is the pathway that leads to imminent capsize via the saddle. The trajectory that leads to escape from the potential well and corresponds to the imminent capsize of a ship, for example the red trajectory, lies inside the tube. Similarly, a trajectory that does not escape the potential well, for example the blue trajectory, and corresponds to an upright ship, lies outside the tube. These example trajectories are shown here in the $x-y-v_y$ space for the energy $e = 0.28$, and also as projection in the configuration space $x-y$. (b) Tube manifolds associated with the left and right saddles in the $x-y-v_y$ space for a given energy $e$. The green and red surfaces denote the stable and unstable tube manifolds, respectively.}
	\label{fig:3D_tubes_mod7_v1}
\end{figure}
The ($y-v_y$) SOS captures the escape trajectories when the energy is greater than the critical energy as shown in Fig.~\ref{fig:sosU1_diff_energy}.  

\textbf{b.} We obtain the periodic orbit associated with the rank-1 saddle using differential correction. The periodic orbit (p.o.) corresponding to a energy level $e + \Delta e$ is the bounding p.o., and projects as a line on the configuration space, $(x,y)$ as shown in Fig.~\ref{fig:3D_tubes_mod7_v1}(a).

\textbf{c.} We compute the invariant manifolds associated with the p.o. of energy, $e + \Delta e$, using globalization and numerical continuation for the left and right saddle equilibrium points, and as shown in Fig.~\ref{fig:3D_tubes_mod7_v1}(b). The invariant manifolds being codimension-1 in the 3-dimensional energy surface, and having geometry $\mathbb{R}^1 \times \mathbb{S}^1$ are cylindrical, and hence are referred to as \emph{tube manifolds}~\cite{Koon2000a}. 
\begin{figure}[!ht]
	\centering
	\subfigure[]{\includegraphics[scale=0.25]{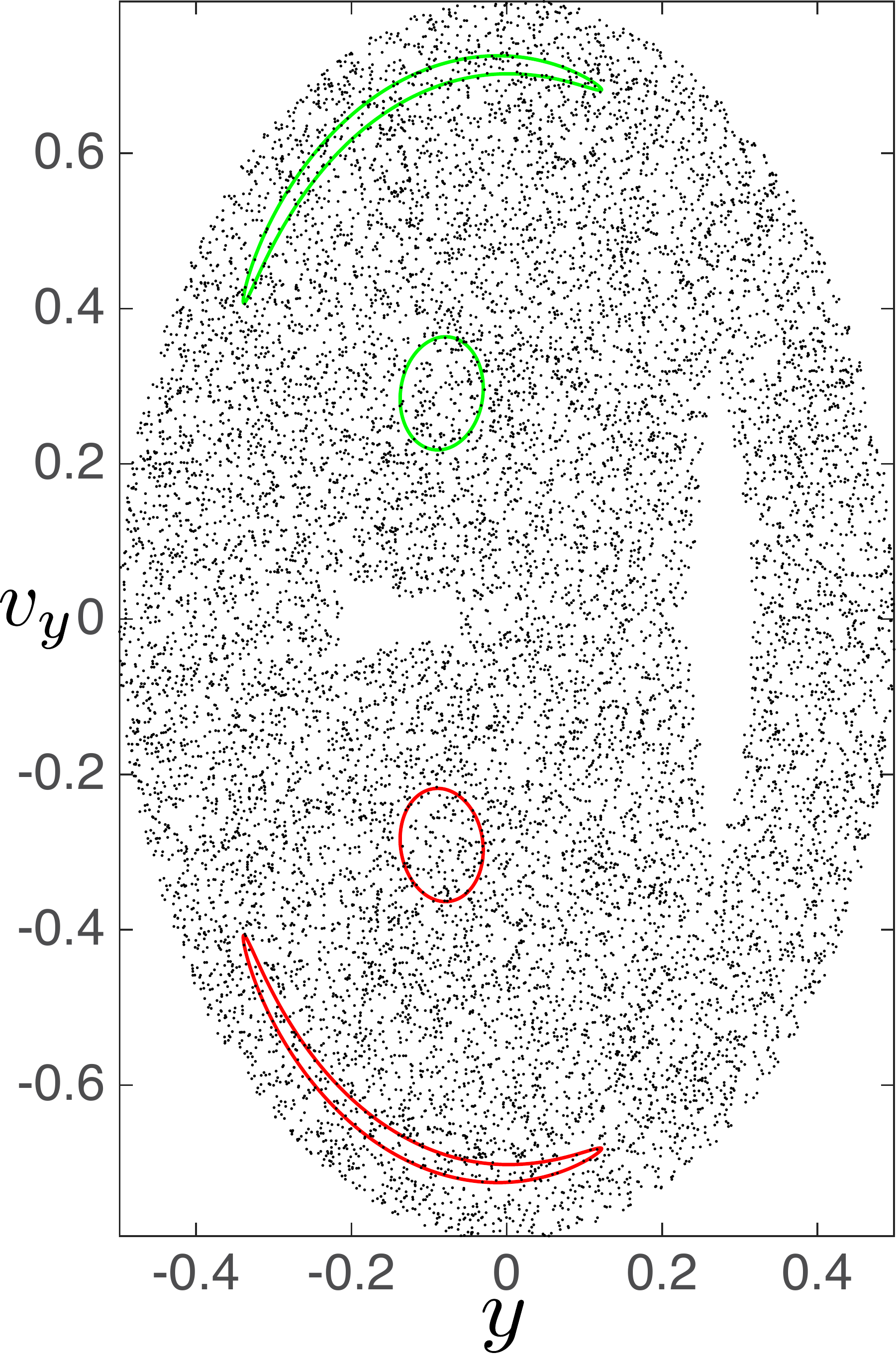}}
	\subfigure[]{\includegraphics[scale=0.26]{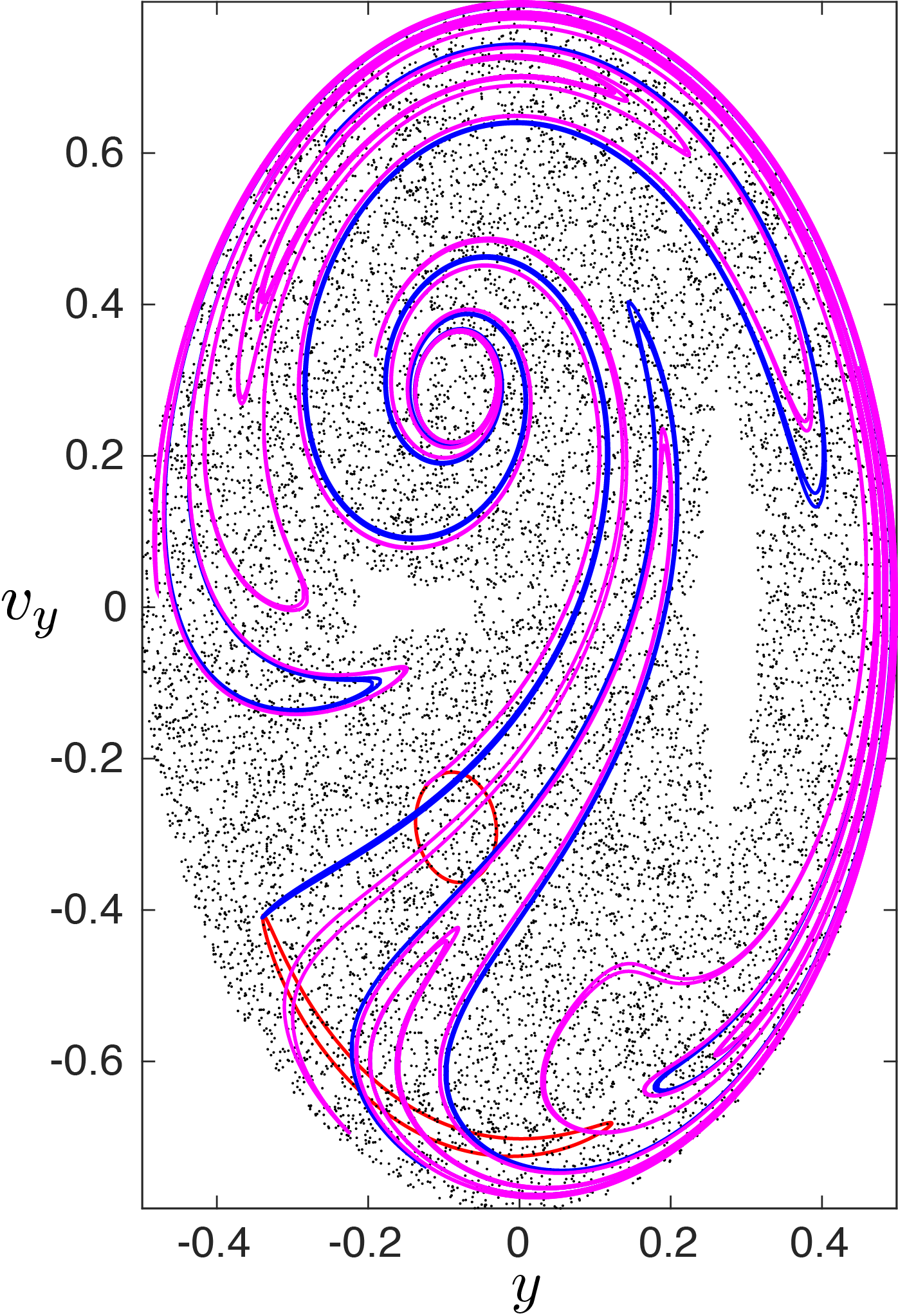}} 
	\caption{(a) Shows the tube manifolds as closed curves on the Poincar\'e SOS, $U_1$~\eqref{eqn:sos_U1}. The green and red curves denote stable and unstable manifolds, respectively, of right and left saddle equilibrium points. The black dots correspond to the the trajectories at the critical energy that intersect the SOS. (b) Shows the pre-images of the unstable manifolds as magenta curves. We have used $\Delta e=0.00307$ above the critical energy for the escape rate computations.}
	\label{fig:sos_U1_tube}
\end{figure}

\textbf{d.}  We obtain the first intersection of the unstable and stable tube manifolds with the Poincar\'e SOS, Eqn.~\eqref{eqn:sos_U1}, and is shown as red and green curves in the Fig.~\ref{fig:sos_U1_tube}(a). By the geometry of the stable tube manifolds, the first intersection is the boundary of the trajectories that lead to imminent escape, that is, they do not return to the surface and lead to escape via the right or left bottlenecks. Similarly, the unstable tube manifolds lead to imminent entry, that is, the trajectories entering the potential well via the respective bottlenecks and intersecting the SOS. 

Using the boundary of these intersections and the pre-images under the map, Eqn.~\eqref{eqn:return_map2D}, we can partition the SOS into regions with different exit time. This time is measured as the number of intersections with the SOS a trajectory undergoes before exiting or after entering the potential well. Hence, we can calculate what fraction of the set of trajectories will lead to escape by computing the area of intersection of pre-images of the first intersection with itself. We denote the first intersection of the stable tube manifolds with the SOS by $\Gamma^s_r$ and $\Gamma^s_l$, where the subscripts denote exit via the right and left bottlenecks, respectively. The geometry of the manifolds tells us that the trajectories that lead to imminent capsize in $n$ iterates of the return map, Eqn.~\ref{eqn:return_map2D}, must be inside either $\Gamma^s_r$ or $\Gamma^s_l$ in the $n^{\textrm{th}}$ pre-image. Thus, the trajectories that lead to escape in $n$ after starting inside the unstable tube manifold's first intersection will be inside the $n^{\textrm{th}}$ pre-image of the first intersection, say $g^{-n}(\Gamma^u_r)$, and $\Gamma^s_r$. The results of this computation are shown in Table.~\ref{tab:escape_2DOF} as the fraction of trajectories that start inside $\Gamma^u_r$ or $\Gamma^u_l$.

\begin{table}[!ht]
	\centering
	\begin{tabular}{|c|c|c|}
		\hline
		Exit via left after intersection \# & Entrance via left & Entrance via right \\
		\hline
		1 &  0\% & 0\% \\
		2 &  0\% & 11.5\% \\
		3 &  2.93\% & 0.016\% \\
		4 & 1.87\% & 1.441\% \\
		\hline
		Exit via right after intersection \# & Entrance via left & Entrance via right \\
		\hline
		1 &  0\% & 0\% \\
		2 &  0\% & 0\% \\
		3 &  11.2\% & 2.90\% \\
		4 &  0.0246\% & 0.278\% \\
		\hline
	\end{tabular}
	\caption{Percentage of trajectories escaping via left/right stable tube that entered via left and right unstable tubes.}
	\label{tab:escape_2DOF}
\end{table}

The rate of escape computation based on geometry of manifolds is advantageous, and the numerical methods presented herein make some progress in this direction. This becomes much more desirable when considering the diverse applications of this approach to chemical physics~\cite{Jaffe1999}, celestial mechanics~\cite{jaffe2002statistical}, and ship capsize~\cite{Naik2017}. 

\section{Conclusion}
Lobe dynamics describes global transport in terms of lobes, parcels of phase space bounded by stable and unstable invariant manifolds associated to hyperbolic fixed points of the system. 
Escape from a potential well describes the phase space structures that lead to critical events in dynamical systems by crossing of partial barriers around rank-1 saddle equilibrium points.
Both of these frameworks---in the circumstances where the dynamics can be reduced to two-dimensional maps---require computation of curves with high density of points, their intersection points, and area bounded by these curves to quantify phase space transport.  

In this article, we developed methods for quantifying phase space transport that requires computing area bounded between curves. In general, these curves may be obtained from a higher dimensional finite-time dynamical system by performing a suitable reduction to two-dimensional map. We developed this method by using a theory for classification of intersection points that involves partitioning of the intersection points into equivalence classes. This enabled application of the discrete form of Green's theorem to compute the area bounded between the segments of the curves. An alternate method for curves with non-transverse intersections---related to identifying lobes by iterating a boundary parametrized by the boundary intersection point---was also presented along with a method to insert points in the curve for increasing the density of points near intersections. This code for increasing the resolution of a curve is implemented in \emph{Curve\_densifier} using an interpolation and insertion method developed for contour surgery in Ref.~\cite{Dritschel1988}, and is made available as open-source repository in Github (\url{https://github.com/Shibabrat/curve_densifier}). In the case of non-transverse intersection, the notion of primary intersection point and secondary intersection point is not suitable and hence, we propose a generalization of the notion of lobes for two intersecting curves using the concept of interior function. The software package \emph{Lober} based on the theory and numerical methods presented in \S~\ref{sect:curve-area},\ref{sect:inter-pts-lobes}, and \ref{sect:non-trans-inter} is made available as an open-source repository in Github (\url{https://github.com/shibabrat/lober}).
We applied these methods for computation of intersection points between two closed curves, lobes defined by these intersections, and the lobe areas to problems in chaotic transport in fluid flow and escape from a potential well in the context of ship capsize---which are reducible to transport in two-dimensional maps. 
Furthermore, the methods presented in this article can also be used in conjunction with homotopic lobe dynamics technique that is an extension of transport in two-dimensional maps to higher dimensional systems in Ref.~\cite{Maelfeyt2017}. 
\section{Acknowledgements}
%
This work was supported in part by the National Science Foundation under awards 1150456, 1520825, and 1537349. The authors would like to thank the anonymous reviewers for their constructive and fruitful suggestions.

\bibliographystyle{plainnat}
\bibliography{refs}

	\appendix
	\setcounter{equation}{0}
	\numberwithin{equation}{section} 
	\numberwithin{figure}{section}
	\renewcommand{\thesection}{Appendix \Alph{section}}
	\section{Usage and outputs}\label{sect:usage}
	This section provides low level details for using the package \textit{Lober}, and processing the output of intersection points, lobes, and lobe areas. There are two primary numerical methods that are implemented in the package, and which are derived in \S\ref{sect:inter-pts-lobes} and \S\ref{sect:non-trans-inter}. These two methods deal with transverse and non-transverse intersections, and computes the intersection points and area enclosed by the piecewise linear segments. The transverse intersection module is activated when the syntax 
	\begin{verbatim}
	lober <c1> <c2> <rslt> [ -DENS <nPass> <nDens> ]
	\end{verbatim} 
	is entered at the terminal. While, the module for non-transverse intersection is activated when {\em -light} is added to the syntax is the form 
	\begin{verbatim}
	lober -light <c1> <c2> <rslt> [ -DENS <nPass> <nDens> ]
	\end{verbatim} 
	where \verb+<c1>+ and \verb+<c2>+ are the names of the files containing the curves $C_1$ and $C_2$, respectively, and \verb+<rslt>+ is the file to be created by \emph{Lober} to save the output. The optional arguments of the built-in \emph{densifier} (\verb+-DENS <nPass> <nDens>+) has been described in the \S\ref{subsubsect:densifier}. The input curves need to be specified in the files in a {\em Tecplot} ASCII format given by
	\begin{verbatim}
	VARIABLES=''x''''y''
	ZONE T=''the curve C1''
	0.2    0.4
	0.23   0.45
	0.35   0.35
	...
	\end{verbatim}
	The output file (\verb <rslt>) contains one line with 4 numbers: the area of the lobes inside, the area of the lobes outside and the relative error on these two values. This is useful to get an estimate of the error involved in computing the area, and provides a first order check of the output. In addition to this, \textit{Lober} generates 6 output files in {\em Tecplot} ASCII format with one header line and points arranged in $N \times 2$ array of $N$ points in $\mathbb{R}^2$. The intersection points are stored in files \verb c10.dat and \verb c20.dat . The set of points which are on the boundary of $C_1 \cap C_2$ and $C_2 \cap C_1$ are stored in \verb c11.dat and \verb c22.dat, respectively, and those on the boundary of $C_1 \setminus C_2$, $C_2 \setminus C_1$, are stored in \verb c12.dat and \verb c21.dat, respectively.
	

\end{document}